\documentclass[10pt,reqno]{amsart}

\usepackage{amsmath,amssymb,amsfonts,amsthm,amsopn,mathtools}
\usepackage{mathrsfs,pifont}
\usepackage{bm}

\usepackage{graphicx}
\usepackage{epsfig}
\usepackage{epstopdf}
\usepackage{color}
\usepackage[dvipsnames]{xcolor}
\usepackage{cases}
\usepackage{booktabs,multirow}
\usepackage{fancybox}
\usepackage{version}

\usepackage{caption} 
\usepackage{subfig} 

\usepackage[ruled,vlined]{algorithm2e}
\usepackage{algorithmic}

\usepackage[marginparwidth=2.5cm, marginparsep=3mm]{geometry}

\usepackage{xr-hyper}
\usepackage{url,hyperref}

\numberwithin{equation}{section}

\theoremstyle{plain} 
\newtheorem{theorem}{Theorem}[section]     
\newtheorem{lemma}[theorem]{Lemma}         

\theoremstyle{definition} 

\theoremstyle{remark} 
\newtheorem{remark}[theorem]{Remark}

\newcommand{\curl}{{\rm curl\,}}

\definecolor{rot}{rgb}{0.000,0.000,0.000}
\definecolor{r1}{rgb}{1.000,0.000,0.000}
\allowdisplaybreaks[4]

\ifpdf
\DeclareGraphicsExtensions{.eps,.pdf,.png,.jpg}
\else
\DeclareGraphicsExtensions{.eps}
\fi

\begin{document}
\baselineskip 14pt
\bibliographystyle{plain}

\title[RCL method for elastic waves] {A Robust Helmholtz-Decomposition-Based Real Compressed Layer Method for Time-Harmonic Elastic Wave Scattering}
\author[
    L.-L. Wang,\;  L. Zhang
	]{Li-Lian Wang${}^{\dag}$ and\; Lu Zhang${}^{\ddag}$
		}
	\thanks{${}^{\dag}$Division of Mathematical Sciences, School of Physical and Mathematical Sciences, Nanyang Technological University, 637371 Singapore. The research of both  authors is partially supported by  Singapore MOE AcRF Tier 2 Grant: MOE-T2EP20224-0012. Email: lilian@ntu.edu.sg.\\	
	\indent ${}^{\ddag}$Division of Mathematical Sciences, School of Physical and Mathematical Sciences, Nanyang Technological University, 637371 Singapore. Email: luzhang@ntu.edu.sg.}


\begin{abstract} Time-harmonic elastic wave  scattering involves both compressional (P-) and shear (S-) waves, which propagate with different wavenumbers and polarization characteristics. The naive construction of perfectly matched layer (PML)-type methods based on complex coordinate stretching may lack robustness, or even fail, particularly when the wavenumbers are highly contrasted. The recently developed real compressed layer (RCL) technique build upon real compression transformations  and explicit extraction of resulting oscillatory patterns for time-harmonic Helmholtz problems may not work, since  the oscillations cannot be explicitly extracted by a single change of variables. This paper intends to bridge this gap by developing  a robust RCL method for two-dimensional time-harmonic elastic wave scattering in unbounded domains with compactly supported inhomogeneities. A key observation is that, through the Helmholtz decomposition, the displacement field in the exterior homogeneous region decoupled into P-wave and S-wave and each has a distinctive separation of its oscillatory pattern and decaying behaviours in polar coordinates. We then apply the real compression coordinate transformation in the radial direction to each component. We further propose a coupled displacement–potential RCL formulation that seamlessly integrates the Helmholtz-decomposed wave components with the interior displacement field. We show that, under this framework, the essential oscillations in the layer can be effectively removed. We prove the well-posedness of the resulting coupled problem and establish the exponential convergence of the RCL solution to the original scattering solution in the truncated domain of interest. We discretize the RCL-system using high-order spectral element method and  demonstrate the effectiveness and robustness  of the proposed method through ample numerical results. 
\end{abstract}
 \maketitle

\section{Introduction}
\label{sec:1}

Time-harmonic elastic wave scattering in inhomogeneous media and unbounded domains arises in many applications, such as geophysical exploration, seismic imaging, nondestructive testing, and material characterization~\cite{A73,HLHCS20,R14}. Unlike acoustic scattering governed by a scalar Helmholtz equation, elastic scattering is modeled by the Navier equation for the displacement field and contains both compressional and shear waves with different wavenumbers and directions. The coupling of two waves, the inhomogeneity of the medium, and the radiation condition at infinity make the problem challenging in both analysis and computation. A key first step is to reduce the unbounded physical domain to a bounded computational domain where standard finite-domain solvers can be used. The truncated problem should be well posed and provide an accurate approximation to the original solution in the region of interest. This motivates the development of efficient and robust domain truncation techniques for elastic wave scattering problems. 

Over the past decades, a variety of techniques have been developed to treat unbounded domains for wave equations~\cite{G13}. One common approach is to truncate the domain by imposing absorbing boundary conditions (ABCs) on an artificial boundary~\cite{BT80,EM77,G91,H99,HL07}. Local ABCs are relatively easy to implement, but they may suffer from limited accuracy and spurious reflections. Alternatively, transparent boundary conditions, such as Dirichlet-to-Neumann maps~\cite{G00,GK96,GK98,KG89}, can represent the radiation condition exactly, but they are typically available only for special geometries and involve nonlocal series representations, leading to dense operators or expensive evaluations. Boundary integral equation methods~\cite{BXY17,BXY191,ZXY21,ZXY22} provide another powerful tool for exterior scattering problems: since the fundamental solution already satisfies the radiation condition, one only needs to reformulate the PDE as an integral equation on the boundary. However, this approach relies on the availability of fundamental solution and becomes less convenient for inhomogeneous media, variable coefficients and volumetric sources. Infinite element methods~\cite{DS06,H16,NW22} encode the outgoing behavior into specially designed basis functions, but their construction and implementation are often problem dependent.

Among various domain truncation techniques, the perfectly matched layer (PML), originally introduced by Bérenger~\cite{B94} for time-domain electromagnetic waves, is one of the most widely used methods due to its flexibility and compatibility with standard finite-element and spectral-element discretizations. The basic idea of PML is to surround the physical computational domain with an artificial layer of finite thickness, filled with specially designed lossy media, so that outgoing waves decay rapidly before reaching the outer boundary. In the frequency domain, the construction of PMLs can be interpreted as a complex coordinate stretching, or equivalently as the introduction of anisotropic absorbing media~\cite{CMH15,MDG14,RGB10}. Since its introduction, PML has been extensively developed for many wave propagation and scattering problems and in various geometrical and coordinate settings~\cite{M21,M22,MMM25}.

Assume that all inhomogeneities and scatterers are enclosed in the disk $B_R$. Observe that in the polar coordinates $(r,\theta)$, the far-field outgoing elastic wave can be decomposed into compressional and shear parts (see the related analysis in subsection \ref{sec:2.2} below),
\begin{equation}
\bm{u}(r, \theta) \sim r^{-1 / 2} e^{{\rm i} k_{\rm{p}} r} \bm{M}_{\rm{p}}(k_{\rm{p}} r, \theta)+r^{-1 / 2} e^{{\rm i} k_{\rm{s}} r} \bm{M}_{\rm{s}}(k_{\rm{s}} r, \theta),
\end{equation}
where $\bm M_{\rm t}$, $\rm t=p,s$, are well behaved. Under the complex stretching 
\begin{equation*}
r\to \tau=r+{\rm i} \int_R^r \sigma(t) d t, \quad R<r<R+d, 
\end{equation*}
each phase factor satisfies
\begin{equation*}
e^{{\rm i} k_{\rm t} \tau(r)}=\exp \Big(-k_{\rm t} \int_R^r \sigma(t) d t\Big) e^{{\rm i} k_{\rm t} r}, \quad t=\rm p,s,
\end{equation*}
where $\sigma(t)$ is called the absorbing function (ABF). One typical choice of the ABF is
\begin{equation}
\sigma(t)=\sigma_0(t-R)^n / d^n, \quad n=0,1, \ldots,
\end{equation}
where $\sigma_0>0$ is a tuning parameter. Then it can be seen from~\cite{BPT10,CXZ16} that the truncated PML solution converges exponentially to the original elastic scattering solution as the layer thickness $d$ or absorbing strength $\sigma_0$ increases. Specifically, for a given truncation accuracy threshold $\epsilon$, the parameters must satisfy
\begin{equation*}
e^{-k_{\rm t}\sigma_0d}\geq\epsilon\Rightarrow \sigma_0d\leq \tfrac{|\ln(\epsilon)|}{k_{\rm t}}.
\end{equation*}
From this estimate, the limitations of PML for elastic waves can be understood from three aspects. First, when $k_t \rightarrow 0$, the decay factor becomes weak, and achieving a given tolerance requires increasing either the layer thickness $d$ or the absorbing strength $\sigma_0$. The former increases the computational cost, while the latter may lead to sharp coefficient variations and discretization-induced reflections. Second, when the two elastic wavenumbers are highly separated, for instance in the case $\lambda \gg \mu$ where $k_{\rm{s}} \gg k_{\rm{p}}$, it is difficult for one set of PML parameters to absorb both P and S waves efficiently. The damping profile needed for the long P wave may be too steep relative to the short $S$ wave, which can generate numerical reflections. Third, when imaginary wavenumbers occur, the associated evanescent waves are not attenuated in the same way as propagating waves and may decay slowly in low-frequency or near-field regimes. Conventional PMLs may therefore fail to absorb these components accurately. 

Several improvements of PML have been developed to enhance its accuracy and robustness. PMLs with unbounded absorbing functions~\cite{BHPR071,BHPR07} have stronger attenuation and may reduce the need for parameter tuning, but they introduce singular coefficients into the transformed equations, which must be handled carefully in numerical discretization. The perfect absorbing layer (PAL) method~\cite{YWG21} provides a further improvement by combining unbounded ABFs mappings with a suitable variable substitution. Consequently, the PAL-equation is free of singular
coefficients and the substituted unknown field is essentially non-oscillatory in the layer. The numerical results show that, although PAL is exact in the continuous setting, achieving the desired accuracy still requires the empirical condition $kd\in[5,30]$. These motivates the search for layer techniques that are easier to implement and less sensitive to parameter choices.

Recently, a RCL method has been developed as a real-coordinate alternative for acoustic scattering problems~\cite{WWW25}. The main difficulty of a real coordinate transformation is that a direct compression of an unbounded domain into a bounded layer produces infinitely fast oscillations, which cannot be resolved by standard grid-based methods. The key idea of RCL is to avoid approximating these oscillations directly. Instead, the outgoing wave is compressed only in the radial direction, and the induced oscillatory phase is explicitly extracted by a suitable substitution. More precisely, with the real-valued exponential mapping $\tau(r)=R e^{\tau_0(r-R)}$, the algebraically decaying factor $r^{-1 / 2}$ is transformed into $\tau(r)^{-1 / 2}=R^{-1 / 2} e^{-\tau_0(r-R) / 2}$, and hence becomes exponentially decaying in the layer. In addition, unlike PML and PAL methods based on complex coordinate transformations, RCL is not merely artificial: the computed field in the layer can be transformed back to recover the outgoing wave in the original unbounded domain. In~\cite{z26}, a related null infinity layer method has also been proposed, which follows a similar idea of removing the far-field decay and oscillations before compactification to make real-coordinate mappings effective for wave scattering problems. However, the method relies on a suitable matching between compactification and phase scaling; for variable-coefficient or strongly coupled systems, additional assumptions and careful discretization near the compactified boundary are still required.

Extending the RCL method to elastic wave scattering is not straightforward. In the acoustic case, the field is governed by a single scalar Helmholtz equation and has one oscillatory phase. In contrast, the elastic field is described by a coupled vector system and contains both compressional and shear waves with different wavenumbers. Moreover, in an inhomogeneous medium, the compressional and shear modes do not satisfy globally decoupled Helmholtz equations in the displacement formulation, and hence cannot be separated by a fixed-wavenumber potential decomposition. To overcome this difficulty, we keep the original elastic wave equation in the inhomogeneous physical region and place the RCL in the exterior homogeneous medium. There, the displacement field can be decomposed by the Helmholtz decomposition into two scalar potentials corresponding to the compressional and shear waves. Then we formulate RCL equations for the compressional and shear potentials under the radial RCT, and introduce separate wavenumber-dependent substitutions to extract the corresponding oscillatory patterns. This leads to a coupled formulation consisting of the elastic wave equation in the inhomogeneous region and two RCL-transformed Helmholtz equations in the compressed layer. It is worth noting that, given a prescribed absorption threshold $\epsilon$, the RCL method achieves nearly wavenumber-independent absorption accuracy as long as the parameters satisfy $\tau_0d\le 2|\ln(\epsilon)|$. Hence, the RCL method is not only effective for high-frequency propagating waves but is also particularly suitable for low-frequency, imaginary-wave-number, and multi-wave-number regimes. This feature is especially relevant to more complicated coupled systems, such as thermoelastic waves, poroelastic waves, biharmonic problems, and time-domain formulations, where imaginary or complex wave numbers and evanescent modes naturally arise. The main contributions of this paper are summarized as follows.\smallskip
\begin{itemize}
	\item We derive the polar far-field representation of two-dimensional elastic outgoing waves and identify that two distinct wavenumber-dependent oscillatory patterns associated with the compressional and shear components. We then introduce the Holmholtz decomposition to effectively manipulate two types of waves which serves as  the basis for this  construction.  
	\smallskip
	
	\item We propose a coupled displacement-potential RCL method for inhomogeneous elastic scattering. The displacement equation is retained in the physical domain, while Helmholtz potential decomposition is introduced in the RCL, where wavenumber-matched changes of variables are used to separate the oscillatory factors induced by real compression. The physical field and the layer potentials are then coupled through the continuity of displacement and traction across the interface.
	
	\smallskip
	\item We prove the well-posedness of the RCL problem and the exponential convergence of the RCL solution as either the absorbing parameter or the thickness of the RCL increases. Then we derive a weak formulation, and develop a high-order spectral element discretization.
\end{itemize}
\smallskip

The rest of this paper is organized as follows. In section \ref{sec:2}, we demonstrate the essential idea for the circular RCL method. Section \ref{sec:3} derives the transformed RCL equations with a general star-shaped domain truncation and conduct the convergence analysis. Section \ref{sec:4} derives the weak formulation of the coupled displacement-potential system. A variety of numerical examples are presented in Section \ref{sec:5} to illustrate the performance of our method.

\section{Preliminaries}
\label{sec:2}

\subsection{Elastic problems}
\label{sec:2.1}

Let $\Omega\in\mathbb R^2$ be a bounded domain with a Lipschitz boundary $\Gamma$ and $\Omega^c=\mathbb{R}^2 \backslash \bar{\Omega} \subset \mathbb{R}^2$ be the exterior unbounded domain. Denote by $\bm\nu=\left(\nu_1, \nu_2\right)^{\top}$ and $\bm\tau=\left(\tau_1, \tau_2\right)^{\top}$ the unit normal and tangential vectors on $\Gamma$, respectively, where $\tau_1=-\nu_2$ and $\tau_2=\nu_1$. We consider the frequency-domain elastic problem:
\begin{equation}
\label{OP}
\begin{dcases}
\Delta^* \bm{u}(\bm x)+\rho\,\omega^2\theta(\bm x) \bm{u}(\bm x)=\bm f(\bm x), & \bm x\in \Omega^c,\\
\bm u(\bm x)=\bm g(\bm x), &\bm x\in\Gamma,
\end{dcases}
\end{equation}
where $\Delta^*$ is the operator defined by
\begin{equation*}
\Delta^*=\mu \Delta+(\lambda+\mu) \nabla \nabla\cdot.
\end{equation*}
Here, $\rho$ denotes the density of the homogeneous medium, while the medium's nonhomogeneity is characterized by the coefficient $\theta(\bm x)$, which satisfies $\theta(\bm x)>0$ for all $\bm x \in \Omega^c$. Let $\Omega_0$ be chosen sufficiently large such that $\Omega\in\Omega_0$ and both $\bm f$ and $(1-\theta)$ are compactly supported in $\Omega_0$. Furthermore, we define $\Omega_1=\Omega_0\backslash\bar\Omega$. Additionally, we denote the angular frequency by $\omega > 0$ and the Lamé constants by $\lambda$ and $\mu$ ($\lambda + \mu > 0$). 

To complete the formulation, the Kupradze radiation condition~\cite{KGBB79} is imposed, which involves decomposing the field $\bm u$ in $\Omega_1^c = \mathbb{R}^2 \setminus \bar{\Omega}_1$ as
\begin{equation}
\label{ups}
\bm{u}=\bm{u}_{\rm p}+\bm{u}_{\rm s}.
\end{equation}
Here, $\bm{u}_{\rm p}$ and $\bm{u}_{\rm s}$ denote the compressional and shear (P and S) waves, respectively,  given by 
\begin{equation*}
\bm{u}_{\rm p}=-\tfrac{1}{k_{\rm p}^2} \nabla \nabla\cdot \bm{u}, \quad \bm{u}_{\rm s}=\tfrac{1}{k_{\rm s}^2}  \bm \curl \curl \bm{u},
\end{equation*}
where 
\begin{equation*}
\bm \curl=(\tfrac{\partial}{\partial x_2},-\tfrac{\partial}{\partial x_1})^{\top}, \quad \curl \bm{v}=\tfrac{\partial v_2}{\partial x_1}-\tfrac{\partial v_1}{\partial x_2}, \quad \bm{v}=(v_1, v_2)^{\top} ,
\end{equation*} 
and the wave numbers $k_{\rm p}$ and $k_{\rm s}$ are defined as
\begin{equation*}
k_{\rm p}=\omega \sqrt{\tfrac{\rho}{\lambda+2 \mu}},\quad k_{\rm s}=\omega \sqrt{\tfrac{\rho}{\mu}}.
\end{equation*}
It is easy to verify that $\bm u_{\rm p}$ and $\bm u_{\rm s}$ satisfy the Helmholtz equations
\begin{equation}
\label{Dcop}
\Delta \bm{u}_{\rm p}+k_{\rm p}^2 \bm{u}_{\rm p}=0, \quad \Delta \bm{u}_{\rm s}+k_{\rm s}^2 \bm{u}_{\rm s}=0 \quad \text { in } \Omega_1^c.
\end{equation}
The Kupradze radiation condition~\cite{KGBB79} is given as
\begin{equation}
\label{SC}
\lim _{|\bm{x}| \rightarrow \infty}|\bm{x}|(\tfrac{\partial \bm{u}_{\rm p}}{\partial|\bm{x}|}-{\rm i} k_{\rm p} \bm{u}_{\rm p})=0, \quad \lim _{|\bm{x}| \rightarrow \infty}|\bm{x}|(\tfrac{\partial \bm{u}_{\rm s}}{\partial|\bm{x}|}-{\rm i} k_{\rm s} \bm{u}_{\rm s})=0 .
\end{equation}

Throughout the paper, for any Banach space $X$, we denote the boldfaced letter $\bm{X}=X^2$. Moreover, we assume $\bm g \in\bm H^{1 / 2}(\Gamma)$ to guarantee existence and uniqueness of the solution $\bm u$ of \eqref{OP} and \eqref{SC} in the Sobolev space $\bm H_{\rm {loc }}^1\left(\Omega^c\right)$. Then the existence and uniqueness of the time harmonic elastic wave equation under the Kupradze radiation condition are given by the following lemma~\cite{BPT10,LLW23}.

\begin{lemma}
	\label{lemma21}
	
	Assume that $\theta \in C^{m_1+1}(\Omega^c)$ with $m_1\geq 1$, and satisfies 
	\begin{equation*}
	\Re (\theta(\bm x))>0 ,\quad \Im (\theta(\bm x))\ge0, \quad\bm x \in \Omega^c,
	\end{equation*}
	and that $1-\theta$ has compact support in $\Omega^c$. Then, for any $\bm{f} \in\bm H^1(\Omega^c)'$ and
	$\bm{g} \in\bm H^{1 / 2}(\Gamma)$, the exterior elastic scattering problem together with the Kupradze radiation condition, admits a unique solution $\bm{u} \in\bm H_{\rm{loc}}^1\left(\Omega^c\right)$. Moreover, for any bounded open set $\mathcal{O} \subset \mathbb{R}^2$ satisfying $\bar{\Omega} \cup \operatorname{supp} \bm{f} \cup \operatorname{supp}(1-\theta) \subset \mathcal{O}$, there exists a constant $C>0$, independent of $\bm{f}$ and $\bm{g}$, such that
	\begin{equation*}
	\|\bm{u}\|_{H^1(\mathcal{O} \backslash \bar{\Omega})} \leq C\left(\|\bm{f}\|_{\bm H^1(\Omega^c)'}+\|\bm{g}\|_{\bm H^{1 / 2}(\Gamma)}\right).
	\end{equation*}
	
\end{lemma}

\subsection{Far-field expansion}
\label{sec:2.2}

In this subsection, we consider the two-dimensional elastic scattering problem by a circular obstacle, aiming to establish the expansion form of the solution. 

Let $B_a=\left\{{\bm{x}} \in \mathbb R^2:|{\bm{x}}|<a\right\}$ be a suitable disk that contains the scatterer $\Omega$ with the boundary $\Gamma=\partial B_a$ (see Figure \ref{GeometricSettings} (a)). Then we take $\Omega_1=B_a$. Assume that the field $\bm u$ on boundary $\Gamma_1$ is given as
\begin{equation}
\label{BouData}
\bm g_1(\theta)=g_{\rm r}(\theta)\bm e_{\rm r}+g_{\rm \theta}(\theta)\bm e_{\rm \theta},
\end{equation}
where $\bm e_r$ and $\bm e_\theta$ denote the polar basis vectors. Then the field $\bm u$ is also the solution of the system 
\begin{equation}
\label{ugamma1}
\begin{dcases}
\Delta^* \bm{u}(\bm x)+\rho\omega^2\bm{u}(\bm x)=\bm 0, & \bm x\in \Omega_1^c,\\
\bm u(\bm x)=\bm g_1(\bm x), &\bm x\in\Gamma_1.
\end{dcases}
\end{equation}
For any solution $\bm u$ in $\Omega_1^c$, we introduce the Helmholtz decomposition
\begin{equation}
\label{HD}
\bm u=\nabla \phi_p+\curl \phi_s, 
\end{equation}
where $\phi_p$, $\phi_s$ are scalar potential functions. Substituting \eqref{HD} into \eqref{ugamma1} yields that
\begin{equation}
\label{HDEqu}
\begin{dcases}
\Delta\phi_{\rm t}+k_{\rm t}^2 \phi_{\rm t}=0,& \text{ in } \; \Omega_1^c,\\
\partial_{\bm\nu}\phi_{\rm p}+\partial_{\bm\tau}\phi_{\rm s}=\bm\nu\cdot{\bm g}_1,& \text { on } \;  \Gamma_1,\\
\partial_{\bm\tau}\phi_{\rm p}-\partial_{\bm\nu}\phi_{\rm s}=\bm\tau\cdot{\bm g}_1,& \text { on } \; \Gamma_1,\\
\lim\limits_{r \rightarrow \infty} r^{1 / 2}\left(\partial_r \phi_{\rm t}-k_{\rm t} \phi_{\rm t}\right)=0.
\end{dcases}
\end{equation}
The following lemma establishes the equivalence between the boundary value problems \eqref{OP} and \eqref{HDEqu} (see \cite{LWWZ2016,LY22} for the proof).

\begin{lemma}
	\label{lemma22}
	Let $\bm u$ be the the solution of the boundary value problem \eqref{ugamma1}. Then $\phi_{\rm p}=-\frac{1}{k_{\rm p}^2}\nabla\cdot\bm u$ and $\phi_{\rm s}=\frac{1}{k_{\rm s}^2}\curl\bm u$ are the solutions of the coupled boundary value problem \eqref{HDEqu}. Conversely, if $\phi_p$, $\phi_s$ are the solutions of the boundary value problem \eqref{HDEqu}, $\bm u=\nabla \phi_p+\curl \phi_s$ is the solution of the boundary value problem \eqref{ugamma1}.
\end{lemma}

We can explicitly solve system \eqref{ugamma1} and \eqref{HDEqu}, and the result is given in the following theorem.
\begin{lemma}
	\label{lemma23}
	The solutions $\bm u$ and $\phi_t$ exists the following convergent series expansion: 
	\begin{align}
	\label{usR}
	\bm u(r, \theta)=&H_0^{(1)}(k_{\rm p} r) \sum_{l=0}^{\infty} \frac{\bm F_{{\rm p},l}(\theta)}{(k_{\rm p} r)^l}+H_1^{(1)}(k_{\rm p} r) \sum_{l=0}^{\infty} \frac{\bm G_{{\rm p},l}(\theta)}{(k_{\rm p} r)^l}\\
	\nonumber
	&+H_0^{(1)}(k_{\rm s} r) \sum_{l=0}^{\infty} \frac{\bm F_{{\rm s},l}(\theta)}{(k_{\rm s} r)^l}+H_1^{(1)}(k_{\rm s} r) \sum_{l=0}^{\infty} \frac{\bm G_{{\rm s},l}(\theta)}{(k_{\rm s} r)^l},\\
	\label{phipsR}
	\phi_{\rm t}(r, \theta)=&H_0^{(1)}(k_{\rm t} r) \sum_{l=0}^{\infty} \frac{h_{{\rm t},l}(\theta)}{(k_{\rm t} r)^l}+H_1^{(1)}(k_{\rm t} r) \sum_{l=0}^{\infty} \frac{q_{{\rm t},l}(\theta)}{(k_{\rm t} r)^l},\quad \rm t=p,s,
	\end{align}
	for $r>a$ and the series converges absolutely and uniformly in $r\geq a+\epsilon>a$. The series is infinitely differentiable term-by-term with respect to $r$ and $\theta$, with all resulting series converging absolutely and uniformly. Here, the coefficients $\{\bm F_{{\rm t},l},\bm G_{{\rm t},l},h_{{\rm t},l},q_{{\rm t},l}\}$ in $\theta$ can be determined recursively by the boundary data $\bm g_1$.
\end{lemma}

\begin{proof}
	Details of the proof are given in Appendix \ref{proofoflemma}.
\end{proof}

As is known~\cite{OLBC10}, the two Hankel functions have the following representations
\begin{equation}
\label{HankelEx}
\begin{aligned}
& H_0^{(1)}(z)=(\tfrac{2}{\pi z})^{1 / 2} e^{{\rm i}(z-\pi / 4)}\big(1-\tfrac{1}{8 z} {\rm i}+\tfrac{3}{128 z^2} {\rm i}^2+\cdots\big), \\[4pt]
& H_1^{(1)}(z)=(\tfrac{2}{\pi z})^{1 / 2} e^{{\rm i}(z-3 \pi / 4)}\big(1+\tfrac{3}{8 z} {\rm i}+\tfrac{3}{128 z^2} {\rm i}^2+\cdots\big),
\end{aligned}
\end{equation}
for $-\pi+\delta \leq \arg z \leq 2 \pi-\delta$ with some small $\delta>0$. Substituting \eqref{HankelEx} into \eqref{usR} and \eqref{phipsR} yields
\begin{align}
\label{usR2}
&\bm u(r, \theta)=\sqrt{\tfrac{2}{\pi k_{\rm p} r}} e^{{\rm i} k_{\rm p} r} \bm M_{\rm p}(k_{\rm p} r, \theta)+\sqrt{\tfrac{2}{\pi k_{\rm s} r}} e^{{\rm i} k_{\rm s} r} \bm M_{\rm s}(k_{\rm s} r, \theta),\\
\label{psR2}
&\phi_{\rm t}(r, \theta)=\sqrt{\tfrac{2}{\pi k_{\rm t} r}} e^{{\rm i} k_{\rm t} r} L_{\rm t}(k_{\rm t} r, \theta),\quad \rm t=p,s,
\end{align}
for $r>a$, where
\begin{equation*}
\begin{aligned}
\bm M_{\rm t}(k_{\rm t}r, \theta)& = e^{-\frac{\pi}{4} {\rm i}}\{\bm F_{{\rm t},0}(\theta)+\tfrac{1}{k_{\rm t}r}[\bm F_{{\rm t},1}(\theta)-\tfrac{{\rm i}}{8} \bm F_{{\rm t},0}(\theta)]+\tfrac{1}{k_{\rm t}^2r^2}[\bm F_{{\rm t},2}(\theta)-\tfrac{{\rm i}}{8} \bm F_{{\rm t},1}(\theta) \\
& \quad +\tfrac{3 {\rm i}}{128} \bm F_{{\rm t},0}(\theta)]+\mathcal{O}(k_{\rm t}^{-3}r^{-3})\}+e^{-\tfrac{3 \pi}{4} {\rm i}}\{\bm G_{{\rm t},0}(\theta)+\tfrac{1}{k_{\rm t}r}[\bm G_{{\rm t},1}(\theta)+\tfrac{3 {\rm i}}{8} \bm G_{{\rm t},0}(\theta)] \\[6pt]
& \quad +\tfrac{1}{k_{\rm t}^2r^2}(\bm G_{{\rm t},2}(\theta)+\tfrac{3 {\rm i}}{8} \bm G_{{\rm t},1}(\theta)+\tfrac{3 {\rm i}}{128} \bm G_{{\rm t},0}(\theta))+\mathcal{O}(k_{\rm t}^{-3}r^{-3})\},\quad \rm t=p,s,\\[6pt]
L_{\rm t}(k_{\rm t}r,\theta)&=e^{-\tfrac{\pi}{4} {\rm i}}\{h_{{\rm t},0}(\theta)+\tfrac{1}{k_{\rm t}r}[h_{{\rm t},1}(\theta)-\tfrac{{\rm i}}{8} h_{{\rm t},0}(\theta)]+\tfrac{1}{k_{\rm t}^2r^2}[h_{{\rm t},2}(\theta)-\tfrac{{\rm i}}{8} h_{{\rm t},1}(\theta) \\
& \quad +\tfrac{3 {\rm i}}{128} h_{{\rm t},0}(\theta)]+\mathcal{O}(k_{\rm t}^{-3}r^{-3})\}+e^{-\tfrac{3 \pi}{4} {\rm i}}\{q_{{\rm t},0}(\theta)+\tfrac{1}{k_{\rm t}r}(q_{{\rm t},1}(\theta)+\tfrac{3 {\rm i}}{8} q_{{\rm t},0}(\theta)) \\
& \quad +\tfrac{1}{k_{\rm t}^2r^2}(q_{{\rm t},2}(\theta)+\tfrac{3 {\rm i}}{8} q_{{\rm t},1}(\theta)+\tfrac{3 {\rm i}}{128} q_{{\rm t},0}(\theta))+\mathcal{O}(k_{\rm t}^{-3}r^{-3})\},\quad \rm t=p,s.
\end{aligned}
\end{equation*}

\noindent\underline{\bf Some observations:}
\smallskip

For the elastic wave equation, the outgoing field cannot be represented by a single scalar phase, since it contains both P and S wave components. In the far field, however, the same separation idea still applies mode by mode. More precisely,
\begin{itemize}
	\item[(a)] the P and S wave components decay slowly at the rates $1/\sqrt{k_{\rm p} r}$ and $1/\sqrt{k_{\rm s} r}$, respectively;
	\item[(b)] they exhibit distinct oscillatory patterns $e^{{\rm i} k_{\rm p} r}$ and $e^{{\rm i} k_{\rm s} r}$, respectively;
	\item [(c)] the functions $\bm M_{\rm t}$ and $L_{\rm t}$ essentially have no oscillation.
\end{itemize}
\smallskip
To this end, an effective RCL must absorb both P and S waves while managing their independent oscillatory behaviors. By decoupling the displacement field via Helmholtz decomposition, one can extract the the oscillatory factor of compressional and shear components. This forms the basis for our RCL method designed to handle the elastic wave scattering problems.

\begin{remark}
	The circular geometry is used only for simplicity of presentation. For a general bounded scatterer, one can choose a star-shaped curve $r=r_0(\theta)$ enclosing the scatterer and the inhomogeneous region. For each fixed $\theta$, the exterior field is considered on the ray $r>r_0(\theta)$, where the medium is homogeneous and the field $\bm u$ satisfies the same Navier equation and radiation condition. Therefore, the Helmholtz decomposition and the corresponding series expansion remain valid along each such ray. The circular case corresponds to the special choice $r_0(\theta) \equiv a$. 
\end{remark}

\subsection{Essence of RCL}
\label{sec:2.3}

\begin{figure}[htb]
	\centering
	\begin{tabular}{c@{\hspace{2cm}}c}
		\includegraphics[scale=0.12]{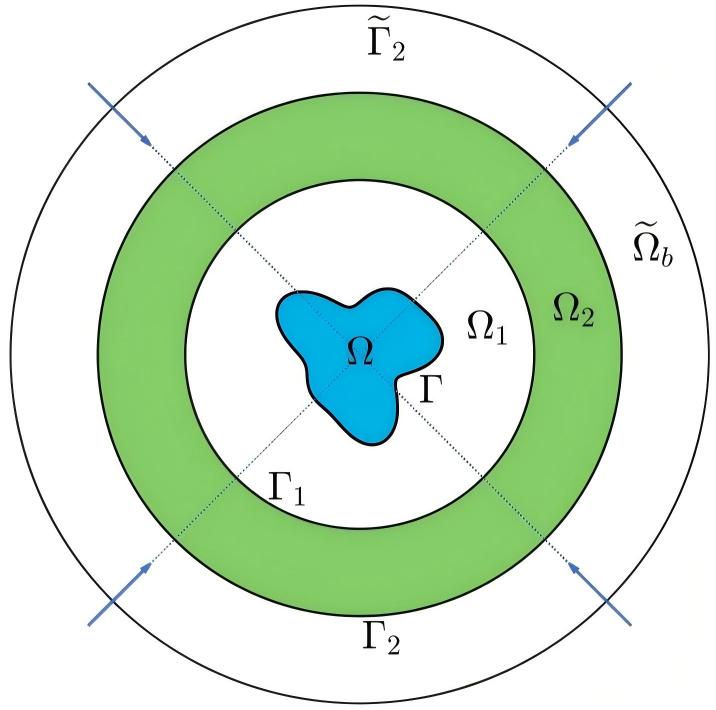} &
		\includegraphics[scale=0.15]{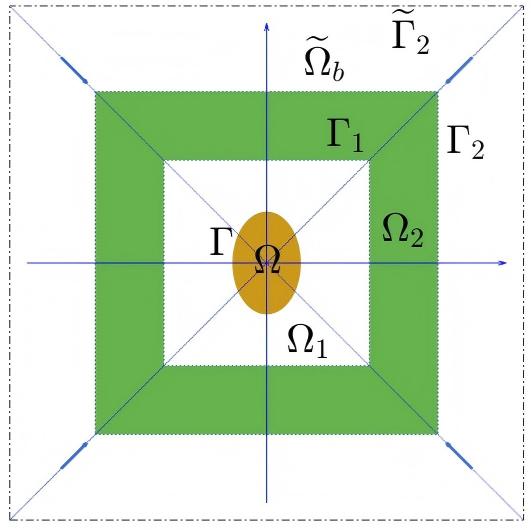} \\
		(a) Circular RCL  &(b) Rectangular RCL
	\end{tabular}
	\caption{Schematic illustration of the artificial layers.}
	\label{GeometricSettings}
\end{figure}

We first introduce a real exponential transform~\cite{WWW25}:
\begin{equation}
\label{RCL}
S=\tau(r,\theta):= \begin{dcases}r & \text { if } r \leq a(\theta), \\ a(\theta) e^{\tau_0(r-a(\theta))} & \text { if } r>a(\theta),\end{dcases}
\end{equation}
where $\tau_0>0$ is a tuning parameter. As shown in Figure \ref{GeometricSettings}, $a$ is constant for a circular $\Gamma_1$, whereas for other shapes, $a$ depends on $\theta$, thereby allowing flexible domain truncation for obstacles of arbitrary geometry.

In view of \eqref{psR2}, the stretched field $\tilde\phi_{\rm t}(r,\theta)=\phi_{\rm t}(\tau,\theta)$ for $r>a$ can be given by
\begin{equation}
\label{phiv}
\tilde\phi_{\rm t}(r,\theta)=\sqrt{\tfrac{2}{k_{\rm t} \pi a}} e^{-\frac{\tau_0}{2}(r-a)} \exp ({\rm i} k_{\rm t} a e^{\tau_0(r-a)})\,L_{\rm t}(k_{\rm t} \tau(r), \theta)=\mathcal{O}(e^{-\frac{\tau_0}{2}(r-a)}),
\end{equation}
which decays exponentially in $r$, and 
\begin{equation*}
L_{\rm t}(k_{\rm t} \tau, \theta) \sim\left\{h_{{\rm t},0}(\theta)-q_{{\rm t},0}(\theta)+{\rm i}\left(h_{{\rm t},0}(\theta)+q_{{\rm t},0}(\theta)\right)\right\} / \sqrt{2}, \quad k_{\rm t}\tau \gg 1 .
\end{equation*}
Whenever $ae^{\tau_0(r-a)} > r$, the oscillations within the neighborhood $r \in (a, a + \delta)$ (for some $\delta > 0$) may increase, which stems from the fact that the exponential mapping compresses the wave's phase. Nevertheless, the high oscillations of $\phi_{\rm p}$ and $\phi_{\rm s}$ can be effectively extracted as
\begin{equation}
\tilde\phi_{\rm t}(r,\theta)=e^{{\rm i}k_{\rm t}(\tau(r)-a)} \tilde v_{\rm t}(r,\theta),
\end{equation}
where $\tilde v_{\rm t}$ decays exponentially without essential oscillations. As shown in \eqref{WF1} of section \ref{sec:4}, based on the Helmholtz decomposition, an equivalent non-oscillatory variational problem can be derived.

As an illustrative example, we consider the scattering of a point source located at the origin by a circular scatterer $\Omega$ with radius $r=2$. The exact solution is given by
\begin{equation*}
\bm u=\nabla H_0^{(1)}(k_{\rm p}r)+\curl H_0^{(1)}(k_{\rm s}r).
\end{equation*}
Figures \ref{Osel} (a) and (b) illustrate the profiles of $u_1(r,\theta)$, $\tilde u_1(r,\theta)=u_1(\tau,\theta)$ for $a=4$, $\tau_0=20$, $\omega=25$, $\lambda=1$, $\mu=2$, $\rho=1$ and $\theta = 0$. It can be seen that $u_1$ and $\tilde u_1$ are identical within the annulus $\Omega_1=\{2< r< a\}$, however, $u_1$ decays slowly while $\tilde u_1$ exhibits rapid decay in $r$ for $r>a$. Notably, $\tilde u_1$ suffers from severe oscillations at $r=a$. To eliminate these oscillations, we first introduce the Helmholtz decomposition and then extract the high-frequency factor. Figures \ref{Osel} (c) and (d) show that, despite the high oscillations of $\tilde\phi_{\rm t}$ near $r=a$, $\tilde v_t$ remains essentially oscillation-free.

\begin{figure}[htb]
	\centering
	\begin{tabular}{cc}
		\includegraphics[scale=0.15]{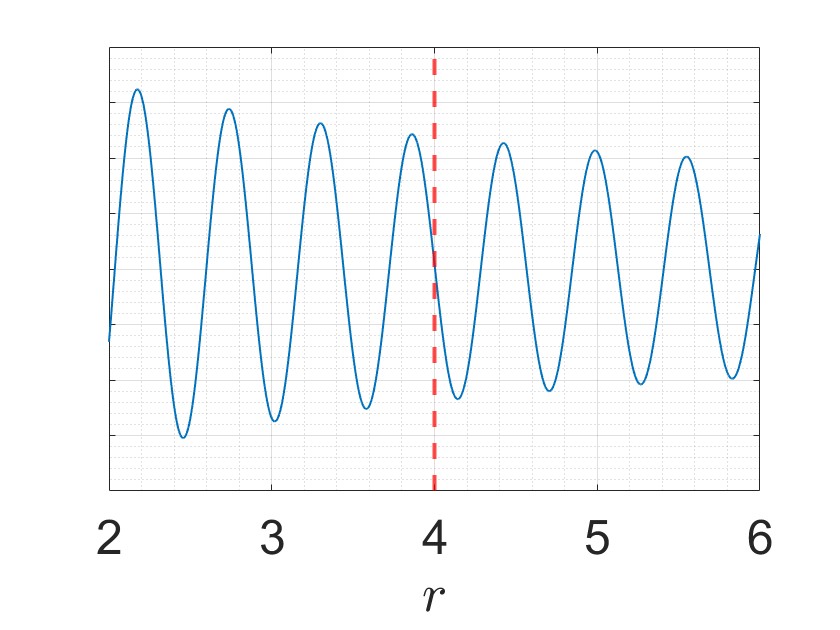} &
		\includegraphics[scale=0.15]{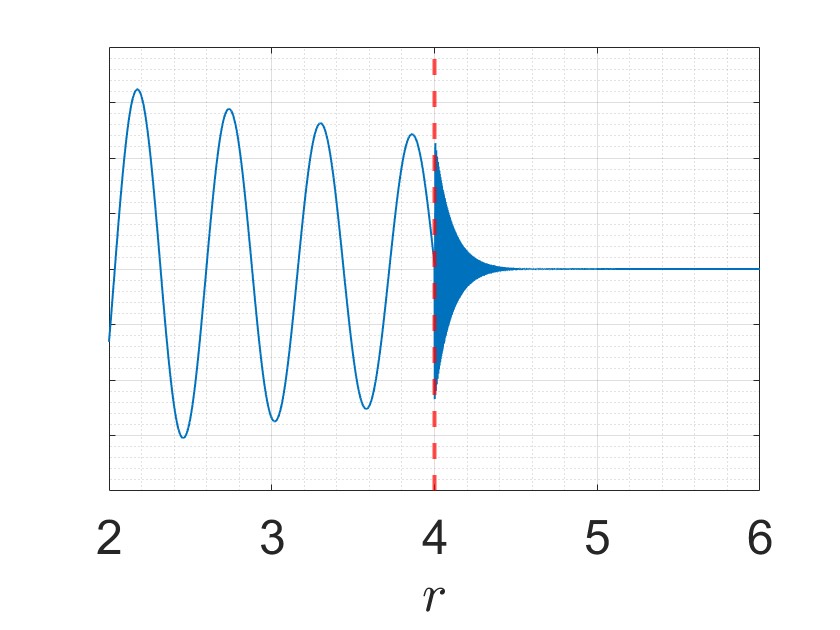} \\
		(a) $\Re(u_1(r,0))$   &(b) $\Re(\tilde u_1(r,0))$\\  
		\includegraphics[scale=0.15]{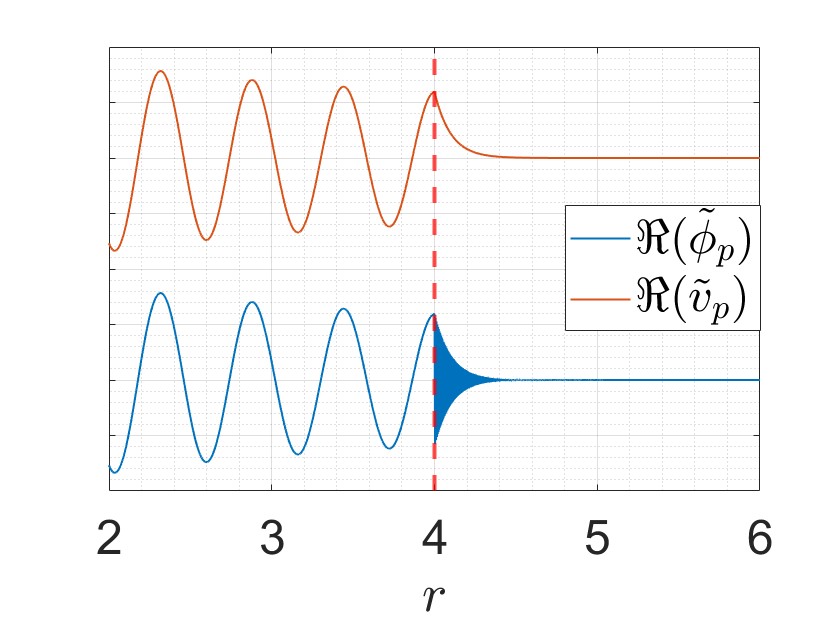} &
		\includegraphics[scale=0.15]{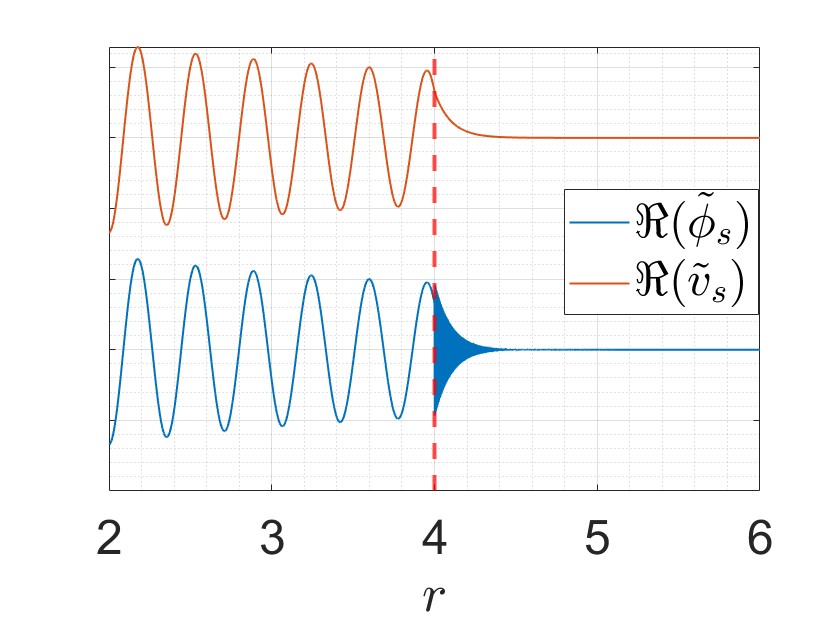} \\
		(c) $\Re\{\tilde\phi_p(r,0),\tilde v_p(r,0)\}$  &(d) $\Re\{\tilde\phi_s(r,0),\tilde v_s(r,0)\}$\\
	\end{tabular}
	\caption{Profiles of the real part of $u_1(r,\theta)$, $\tilde u_1(r,\theta)$, $\tilde\phi_{\rm t}(r,\theta)$ and $v_t(r,\theta)$ with $\theta=0$.}
	\label{Osel}
\end{figure}

\section{Proof of well-posedness and convergence of RCL method}
\label{sec:3}

In this section, we begin by formulating the RCL equations using the RCT. Subsequently, a rigorous analysis regarding the well-posedness and convergence of the proposed method is provided.

\subsection{The real coordinate transformation} 
\label{sec:3.1}

Rewrite the real stretching as
\begin{equation*}
\tilde{\bm x}=(\tilde x_1,\tilde x_2)^\top=(\tau\cos(\theta),\tau\sin(\theta))^\top=\big(\tfrac{\tau x_1}{r},\tfrac{\tau x_2}{r}\big)^\top. 
\end{equation*}
One verifies readily that
\begin{equation}
\label{Hna}
\nabla=\begin{pmatrix}
\partial_{x_1}\\
\partial_{x_2}
\end{pmatrix}=\mathbb R_\theta\widehat{\nabla},
\end{equation}
where $\widehat{\nabla}=(\partial_r,\tfrac{\partial\theta}{r})$ and 
\begin{equation*}
\mathbb R_\theta=\begin{bmatrix}
\cos(\theta)&-\sin(\theta)\\[4pt]
\sin(\theta)&\cos(\theta)
\end{bmatrix}.
\end{equation*}
It is easy to see that the Jacobian matrix is given by 
\begin{equation*}
\mathbb J=\begin{bmatrix}
\tfrac{\partial\tilde x_1}{\partial x_1}&\tfrac{\partial\tilde x_1}{\partial x_2}\\[4pt]
\frac{\partial\tilde x_2}{\partial x_1}&\frac{\partial\tilde x_2}{\partial x_2}
\end{bmatrix}=\mathbb R_\theta \begin{bmatrix}
\tau_r&\tfrac{\tau_\theta}{r}\\[4pt]
0&\tfrac{\tau}{r}
\end{bmatrix}\mathbb R_\theta^\top,\quad \tau_r=\frac{\partial\tau}{\partial r},\quad \tau_\theta=\frac{\partial\tau}{\partial \theta}.
\end{equation*}
Then we have
\begin{equation*}
J=|\mathbb J|=\frac{\tau\tau_r}{r},\quad \mathbb J^{-1}=\mathbb R_\theta\begin{bmatrix}
\frac{1}{\tau_r}&-\frac{\tau_\theta}{\tau_r\tau}\\[4pt]
0&\frac{r}{\tau}
\end{bmatrix}\mathbb R_\theta^\top.
\end{equation*}
We introduce the stretched gradient, divergence, $\curl$ and Laplace operators
\begin{equation*}
\begin{aligned}
&\widetilde{\nabla} v:=\mathbb{J}^{-\top} \nabla v, \quad \widetilde{\nabla} \cdot \bm{u}:=J^{-1} \nabla\cdot( J \mathbb{J}^{-1} \tilde{\bm{u}}), \\
&\widetilde{\Delta} v:=J^{-1} \nabla\cdot(\mathbb{A} \nabla v),\quad \widetilde{\curl}\bm u=J^{-1}\curl(\mathbb P\mathbb J^{\top}\bm u),\quad\widetilde{\bm\curl}v=J^{-1}\mathbb J\bm\curl v,
\end{aligned}
\end{equation*}
where 
\begin{equation*}
\mathbb P= \begin{bmatrix}
0&1\\[4pt]
1&0
\end{bmatrix},\quad\mathbb{A}=J \mathbb{J}^{-1}\mathbb{J}^{-\top}. 
\end{equation*}

Using the coordinate transformation \eqref{RCL}, we can transform the Lam\'e system~\eqref{OP}  into the following form:
\begin{equation}
\label{PMLeq1}
\mu\nabla\cdot(\mathbb A\nabla \tilde{\bm u}(\bm x))+(\lambda+\mu)J\mathbb{J}^{-\top} \nabla(J^{-1} \nabla\cdot( J \mathbb{J}^{-1}\tilde{\bm{u}}))+\rho\omega^2 J\tilde{\bm u}(\bm x)=\bm 0,\quad\bm x\in\Omega_1^c.
\end{equation}
Due to the exponential decay of the solutions $\tilde{\bm u}$ in the RCL region, the infinite domain can be truncated into a bounded domain. As shown in Figure \ref{GeometricSettings}, we define the following subdomains: 
\begin{equation*}
\Omega_2=\left\{\bm x\in\Omega^c,a(\theta)<|r|<b(\theta)\right\},\quad \Omega_b=\Omega_1\cup\Omega_2,\quad\Gamma_2=\left\{|r|=b(\theta)\right\}.
\end{equation*}
Thus, we can impose the homogeneous Dirichlet boundary condition at $\Gamma_b$. With the help of the above notations, we can obtain the RCL problem
\begin{equation}
\label{PMLeq2}
\begin{dcases}
\Delta^* \hat{\bm u}+\rho\omega^2\theta \hat{\bm u}=\bm f, &\text{ in } \Omega_1,\\
\mu\nabla\cdot(\mathbb A\nabla \hat{\bm u})+(\lambda+\mu)J\mathbb{J}^{-\top} \nabla(J^{-1} \nabla\cdot( J \mathbb{J}^{-1} \bm{u}))+\rho\omega^2 J\hat{\bm u}=\bm 0, &\text{ in } \Omega_2,\\
\bm{\hat u}=\bm g,& \text { on } \Gamma,\\
\bm{\hat u}=\bm 0,& \text { on } \Gamma_2.
\end{dcases}
\end{equation}
In addition, we need apply the standard transmission conditions on $\Gamma_1$, namely the continuity of the displacement and traction traces across the interface $\Gamma_1$. In $\Omega_2$, we can also decompose the field $\hat{\bm u}$ as
\begin{equation}
\label{HDEQ}
\hat{\bm u}=\widetilde\nabla\hat\phi_p+\widetilde{\bm\curl}\hat\phi_s.
\end{equation}
Substituting \eqref{HDEQ} into \eqref{PMLeq2} and using the usual transmission conditions at $\Gamma_1$ yields a system equivalent to \eqref{PMLeq2}, as follows:
\begin{equation}
\label{OP2}
\begin{dcases}
\Delta^* \hat{\bm{u}}(\bm x)+\rho\omega^2\theta(\bm x)  \hat{\bm{u}}(\bm x)=\bm f(\bm x), &\text{ in } \Omega_1,\\
\nabla \cdot(\mathbb A(\bm x) \nabla \hat\phi_{\rm t}(\bm x))+k_{\rm t}^2 J\hat\phi_{\rm t}(\bm x)=0, & \text { in } \Omega_2, \\ 
\hat{\bm u}(\bm x)=\bm g(\bm x), &\text { on }\Gamma,\\
-\lambda k_{\rm p}^2\hat\phi_{\rm p}\bm\nu-\mu k_{\rm s}^2\hat\phi_{\rm s}\bm\tau+2\mu\partial_{\bm\nu}\nabla \hat\phi_{\rm p}+2\mu\partial_{\bm\nu}\curl\hat\phi_{\rm s}=\bm T(\partial,\bm\nu)\hat{\bm u},& \text { on } \Gamma_1,\\
\partial_{\bm\nu}\hat\phi_{\rm p}+\partial_{\bm\tau}\hat\phi_{\rm s}=\bm\nu\cdot\hat{\bm u},& \text { on } \Gamma_1,\\
\partial_{\bm\tau}\hat\phi_{\rm p}-\partial_{\bm\nu}\hat\phi_{\rm s}=\bm\tau\cdot\hat{\bm u},& \text { on } \Gamma_1,\\
\tilde\partial_{\bm\nu}\hat\phi_{\rm p}+\tilde\partial_{\bm\tau}\hat\phi_{\rm s}=0,& \text { on } \Gamma_2,\\
\tilde\partial_{\bm\tau}\hat\phi_{\rm p}-\tilde\partial_{\bm\nu}\hat\phi_{\rm s}=0,& \text { on } \Gamma_2,
\end{dcases}
\end{equation}
where $\tilde\partial_{\bm \nu}={\bm \nu}^\top\mathbb A\nabla$, $\tilde\partial_{\bm \tau}={\bm \tau}^\top\mathbb A\nabla$ and $\bm T$ denotes the traction operator given by 
\begin{equation*}
\bm T(\partial,\bm\nu)\hat{\bm u}=\lambda\bm\nu(\nabla \cdot\hat{\bm u})+2\mu \partial_{\bm\nu}\hat{\bm u}-\mu \bm\tau \curl\hat{\bm u}.
\end{equation*}

Note that we consider the RCL transformation in both displacement \eqref{PMLeq2} and Helmholtz-potential formulations \eqref{OP2}. In the analysis, we use the displacement formulation \eqref{PMLeq2}, which preserves the structure of the original elastic system and is suitable for proving convergence. In the computation, we adopt a coupled displacement-potential formulation \eqref{OP2}. This allows the P and S waves to be separated and their oscillatory phases to be extracted independently.

\subsection{Convergence analysis}
\label{sec:3.2}
In this subsection, we establish the well-posedness of the RCL problem and prove that the $\bm H^1$-error between the solution $\hat{\bm u}$ of \eqref{PMLeq2} and the original scattering solution $\tilde{\bm u}$ of \eqref{PMLeq1} decays exponentially as either the RCL absorbing coefficient $\tau_0$ or the thickness of the RCL increases.

Define the domain $\widetilde\Omega_b=\tau(\Omega_b)$. Then we transform $\tilde{\bm u}$ back to the $(\tilde x_1,\tilde x_2)$ coordinates by \eqref{RCL} as $\widetilde{\bm U}(\tau, \theta):=\tilde{\bm u}(r, \theta)$, leading to 
\begin{equation}
\label{PMLOeq}
\begin{dcases}
\Delta_{\tilde{\bm x}}^*\widetilde{\bm U}(\tilde{\bm x})+\rho\omega^2\theta(\tilde{\bm x})\widetilde{\bm U}(\tilde{\bm x})=\bm f, \quad & \tilde{\bm x}\in \widetilde\Omega_b, \\
\widetilde{\bm U}(\tilde{\bm x})=\bm g(\tilde{\bm x}), \quad & \tilde{\bm x}\in \Gamma, \\
\text{Kupradze radiation conditions} \quad & \text {as } \tau \rightarrow \infty.
\end{dcases}
\end{equation}
Assume that $\tilde{\bm u}|_{\Gamma_1}=\bm \xi$. Then $\widetilde{\bm U}$ is also the solution of the exterior problem:
\begin{equation}
\label{DtNEq}
\begin{dcases}
\Delta_{\tilde{\bm x}}^*\widetilde{\bm U}(\tilde{\bm x})+\rho\omega^2\widetilde{\bm U}(\tilde{\bm x})=\bm 0, \quad & \tilde{\bm x}\in \Omega_1^c, \\
\widetilde{\bm U}(\tilde{\bm x})=\bm\xi(\tilde{\bm x}), \quad  & \bm x\in \Gamma_1, \\
\text{Kupradze radiation conditions} \quad 
& \text{as } \tau\rightarrow \infty.
\end{dcases}
\end{equation}
According to~\cite{KGBB79,LWWZ2016,S08}, we have 
\begin{equation}
\label{DtNError}
\|\bm\lambda\|_{\bm H^{-1/2}(\Gamma_1)} \leq C\|\bm{\xi}\|_{\bm H^{1/2}(\Gamma_1)}.
\end{equation}

As is known~\cite{KGBB79}, the boundary value problem \eqref{DtNEq} has a unique solution $\widetilde{\bm U} \in \bm H_{loc}^1(\Omega_1^c)$. According to the potential theory, the solution $\widetilde{\bm U}$ admits the representation
\begin{equation}
\label{SDLR}
\widetilde{\bm U}(\tilde{\bm x})=\mathcal D(\bm{\xi})(\tilde{\bm x})-\mathcal S(\bm\lambda)(\tilde{\bm x}),\quad\tilde{\bm x}\in\Omega_1^c,
\end{equation}
where $\mathcal S$ and $\mathcal D$ denote the single and double layer potentials, respectively, given by
\begin{align}
\label{SPL}
& \mathcal S(\bm \lambda)(\tilde{\bm x})=\int_{\Gamma_1} \mathbb G(\tilde{\bm x}, \tilde{\bm{y}})\bm{\lambda}(\tilde{\bm{y}}) d s_{\bm y},\quad\tilde{\bm x}\in\Omega_1^c,  \\
\label{DPL}
& \mathcal D(\bm{\xi})(\tilde{\bm x})=\int_{\Gamma_1}[T(\partial_{\tilde{\bm y}},\bm{\nu_{\tilde y}}) \mathbb G(\tilde{\bm x}, \tilde{\bm{y}})]^\top \bm\xi(\tilde{\bm{y}}) d s_{\bm y},\quad\tilde{\bm x}\in\Omega_1^c.
\end{align}
Here, $\mathbb G$ is the fundamental solution of the Navier equation \eqref{DtNEq} in $\mathbb R^2$ given by 
\begin{equation}
\label{FS}
\mathbb G(\tilde{\bm x}, \tilde{\bm y})=\tfrac{1}{\mu} \gamma_{k_{\rm s}}(\tilde{\bm x}, \tilde{\bm y})\mathbb I+\tfrac{1}{\rho \omega^2} \nabla_{\bm{\tilde x}} \nabla_{\bm{\tilde x}}^{\top}[\gamma_{k_{\rm s}}(\tilde{\bm x}, \tilde{\bm y})-\gamma_{k_{\rm p}}(\tilde{\bm x}, \tilde{\bm y})], \quad \tilde{\bm x} \neq\tilde{\bm y},
\end{equation}
where $\mathbb I$ denotes the $2 \times 2$ identity matrix, and $\gamma_{k_{\rm t}}$ is the fundamental solution of the Helmholtz equation in $\mathbb{R}^2$ with wave number $k_{\rm t}$, given by
\begin{equation*}
\gamma_{k_{\rm t}}(\tilde{\bm x},\tilde{\bm y})=\tfrac{\rm i}{4}H_0^{(1)}( k_{\rm t}|\tilde{\bm x}-\tilde{\bm y}|), \quad {\rm t}={\rm p, s} .
\end{equation*}
The estimates of the Green's function given below will play an essential role in the subsequent analysis.

\begin{lemma}
	\label{FSE}
	We introduce $\tilde{\bm r}=\tilde{\bm x}-\tilde{\bm y}$ and $\tilde r=|\tilde{\bm r}|$. For any $\tilde{\bm x} \in \Omega_1^c$, $\tilde{\bm y} \in \Gamma_1$ with $k_{\rm t}\tilde r > 1$, we have
	\begin{align}
	\label{GErr}
	&|\gamma_{k_{\rm t}}(\tilde{\bm x}, \tilde{\bm y})| \leq C k_{\rm t}^{-1 / 2}\tilde r^{-1 / 2},\\
	\label{GyErr}
	& \Big|\tfrac{\partial \gamma_{k_{\rm t}}(\tilde{\bm x}, \tilde{\bm y})}{\partial \tilde{y}_j}\Big| \leq C k_{\rm t}\tilde r^{-1 / 2}, \\
	\label{GxyErr}
	& \Big|\tfrac{\partial^2 \gamma_{k_{\rm t}}(\tilde{\bm x}, \tilde{\bm y})}{\partial \tilde{x}_i \partial \tilde{y}_j}\Big| \leq Ck_{\rm t}^{3/2}\tilde r^{-1 / 2},\\
	\label{GxxyErr}
	&\Big|\tfrac{\partial^3 \gamma_{k_{\rm t}}(\tilde{\bm x}, \tilde{\bm y})}{\partial \tilde{x}_l\partial \tilde{x}_i \partial \tilde{y}_j}\Big|\leq C k_{\rm t}^{5/2}\tilde r^{-1/2},\\
	\label{GxxxyErr}
	&\Big|\tfrac{\partial^4 \gamma_{k_{\rm t}}(\tilde{\bm x}, \tilde{\bm y})}{\partial \tilde{x}_m \partial \tilde{x}_l \partial \tilde{x}_i \partial \tilde{y}_j}\Big|\leq C k_{\rm t}^{7/2}\tilde r^{-1/2}
	\end{align}
	for $i, j,l,m,n=1,2$, where $C$ is a positive constant independent of $\tilde{\bm x}, \tilde{\bm y}$, and $k_{\rm t}$. 
\end{lemma}
\begin{proof}
	The proofs of inequalities \eqref{GErr}-\eqref{GxyErr} can be found in \cite{WWW25}. We now present the proofs of inequality \eqref{GxxyErr}-\eqref{GxxxyErr}. It can be seen from~\cite{OLBC10,WWW25} that 
	\begin{equation}
	\label{HankelErr}
	|H_0^{(1)}(k_{\rm t}\tilde r)| \leq |H_1^{(1)}(k_{\rm t}\tilde r)|\leq k_{\rm t}^{-1/2}|\tilde{\bm x}-\tilde{\bm y}|^{-1/2}|H_1^{(1)}(1)|.
	\end{equation}
	Utilizing the formula
	\begin{equation}
	\label{DHankel}
	\tfrac{d H_0^{(1)}(z)}{d z}=-H_1^{(1)}(z), \quad z \tfrac{d H_n^{(1)}(z)}{d z}+n H_n^{(1)}(z)=z H_{n-1}^{(1)}(z),
	\end{equation}
	we can derive that
	\begin{equation}
	\label{GxxyEq}
	\begin{aligned}
	&\tfrac{\partial^3 \gamma_{k_{\rm t}}(\tilde{\bm x},\tilde{\bm y})}{\partial \tilde{x}_l \partial \tilde{x}_i \partial \tilde{y}_j}=\tfrac{{\rm i}k_{\rm t}^2}{4}H_0^{(1)}(k_{\rm t} \widetilde{r})[\tfrac{\delta_{i l} \tilde r_j+\delta_{j l} \tilde r_i+\delta_{i j} \tilde r_l}{\widetilde{r}^2}-\tfrac{4 \tilde r_i \tilde r_j \tilde r_l}{\widetilde{r}^4}]\\
	& -\tfrac{{\rm i}k_{\rm t}H_1^{(1)}(k_{\rm t} \widetilde{r})}{4\widetilde{r}^5}[(k_{\rm t}^2 \widetilde{r}^2-8) \tilde r_i \tilde r_j \tilde r_l+2(\delta_{i l} \tilde r_j+\delta_{j l} \tilde r_i+\delta_{i j} \tilde r_l) \widetilde{r}^2],
	\end{aligned}
	\end{equation}
	where $\delta$ denotes Kronecker delta. Therefore, we have the following estimate:
	\begin{equation}
	\label{GxxyErr2}
	\begin{aligned}
	&\Big|\tfrac{\partial^3 \gamma_{k_{\rm t}}(\tilde{\bm x},\tilde{\bm y})}{\partial \tilde{x}_l\partial \tilde{x}_i \partial \tilde{y}_j }\Big|\leq C[\tfrac{k_{\rm t}^2H_0^{(1)}(k_{\rm t}\tilde r)}{\tilde r}+H_1^{(1)}(k_{\rm t}\tilde r)(k_{\rm t}^3+\tfrac{k_{\rm t}}{\tilde r^2})]\\
	&\leq C(k_{\rm t}^{3/2}\tilde r^{-3/2} +k_{\rm t}^{5/2}\tilde r^{-1/2}+k_{\rm t}^{1/2}\tilde r^{-5/2})\leq C k_{\rm t}^{5/2}\tilde r^{-1/2}.
	\end{aligned}
	\end{equation}
	Next, we consider the term $\tfrac{\partial^4 \gamma_{k_{\rm t}}}{\partial \tilde{x}_m \partial \tilde{x}_l \partial \tilde{x}_i \partial \tilde{y}_j}$. A direct calculation yields
	\begin{equation*}
	\begin{aligned}
	&\tfrac{\partial^4 \gamma_{k_{\mathrm{t}}}}{\partial \tilde{x}_m \partial \tilde{x}_l \partial \tilde{x}_i \partial \tilde{y}_j}= {\rm i} k_{\mathrm{t}}^2[\tfrac{24-k_{\mathrm{t}}^2 \tilde{r}^2}{4 \tilde{r}^6} R_{m l i j}-\tfrac{1}{\tilde{r}^4} S_{m l i j}+\tfrac{1}{4 \tilde{r}^2} T_{m l i j}] H_0^{(1)}\left(k_{\mathrm{t}} \tilde{r}\right) \\
	& +{\rm i} k_{\mathrm{t}}[\tfrac{2\left(k_{\mathrm{t}}^2 \tilde{r}^2-6\right)}{\tilde{r}^7} R_{m l i j}+\tfrac{8-k_{\mathrm{t}}^2 \tilde{r}^2}{4 \tilde{r}^5} S_{m l i j}-\tfrac{1}{2 \tilde{r}^3} T_{m l i j}] H_1^{(1)}\left(k_{\mathrm{t}} \tilde{r}\right),
	\end{aligned}
	\end{equation*}
	where
	\begin{equation*}
	\begin{aligned}
	&R_{m l i j}=\tilde{r}_m \tilde{r}_l \tilde{r}_i \tilde{r}_j, \quad T_{m l i j}=\delta_{m l} \delta_{i j}+\delta_{m i} \delta_{l j}+\delta_{m j}\delta_{l i},\\
	&S_{m l i j}=\delta_{m l} \tilde{r}_i \tilde{r}_j+\delta_{m i} \tilde{r}_l \tilde{r}_j+\delta_{m j} \tilde{r}_l \tilde{r}_i+\delta_{l i} \tilde{r}_m \tilde{r}_j+\delta_{l j} \tilde{r}_m \tilde{r}_i+\delta_{i j} \tilde{r}_m \tilde{r}_l.
	\end{aligned}
	\end{equation*}
	Using the inequality \eqref{HankelErr}, we can obtain that 
	\begin{equation}
	\begin{aligned}
	&\Big|\tfrac{\partial^4 \gamma_{k_{\rm t}}(\tilde{\bm x}, \tilde{\bm y})}{\partial \tilde{x}_m \partial \tilde{x}_l \partial \tilde{x}_i \partial \tilde{y}_j}\Big|\leq C[k_{\rm t}^{7/2}\tilde r^{-1/2}+(k_{\rm t}^{5/2}+k_{\rm t}^{3/2})\tilde r^{-3/2}+k_{\rm t}^{3/2}\tilde r^{-5/2}\\
	&+(k_{\rm t}^{1/2}+k_{\rm t}^{-1/2})\tilde r^{-7/2}]\leq C k_{\rm t}^{7/2}\tilde r^{-1/2}.
	\end{aligned}
	\end{equation}
	This completes the proof. Note that we also have $\tfrac{\partial \gamma_t}{\partial\tilde x_i}=-\tfrac{\partial \gamma_t}{\partial\tilde y_i}$. 
\end{proof}

For the subsequent analysis, we introduce the weighted $\bm H^{1/2}(\Gamma)$ norm~\cite{CZ10,CZ17} defined by
\begin{equation}
\label{H211}
\|\bm u\|_{\bm H^{1 / 2}(\Gamma)}=\Big( \sum_{i=1}^{2} \|u_i\|_{H^{1/2}(\Gamma)}^2 \Big)^{1/2},
\end{equation}
where
\begin{equation}
\label{H212}
\|u_i\|_{\bm H^{1 / 2}(\Gamma)}=(|\Gamma|^{-1}\|u_i\|_{L^2(\Gamma)}^2+|\Gamma|^{-1}|u_i|_{H^{1 / 2}(\Gamma)}^2)^{1/2},
\end{equation}
with 
\begin{equation}
\label{H213}
|u_i|_{H^{1 / 2}(\Gamma)}^2=\int_{\Gamma} \int_{\Gamma} \tfrac{|u_i(\bm{x})-u_i(\bm y)|^2}{|\bm{x}-\bm y|^2} ds_{\bm x}d s_{\bm y}.
\end{equation}

We define the notation $\widetilde\Gamma_2=\left\{|r|=\tau(b(\theta))\right\}$. With the above notations, we can now estimate the single and double layer potentials. 

\begin{lemma}
	\label{SLPE}
	
	Let $\widetilde\Gamma_2$ be a piecewise smooth, simply connected, closed curve in $\Omega_1^c$ satisfying $k_{\rm t} \operatorname{dist}(\widetilde\Gamma_2 ; \Gamma_1)>1$. Then for any $\bm\lambda \in\bm H^{-1 / 2}(\Gamma_1)$, we have 
	\begin{equation}
	\|\mathcal S(\bm\lambda)(\tilde{\bm x})\|_{\bm H^{1 / 2}(\widetilde\Gamma_2)} \leq C \max\big\{k_{\rm s}^{-1/2},k_{\rm s}^{7/2}\big\}\big\{\operatorname{dist}(\widetilde\Gamma_2 ; \Gamma_1)\big\}^{-\frac{1}{2}}\|\bm \lambda\|_{\bm H^{-1 / 2}(\Gamma_1)},
	\end{equation}
	where $C$ is a positive constant independent of $k_{\rm t}$ and $\bm\lambda$.
\end{lemma}
\begin{proof}
	We rewrite the fundamental solution \eqref{FS} as
	\begin{equation*}
	\mathbb G(\tilde{\bm x},\tilde{\bm y})=\begin{bmatrix}
	G_{11}&G_{12}\\[4pt]
	G_{21}&G_{22}
	\end{bmatrix}=\begin{bmatrix}
	\frac{\gamma_{k_{\rm s}}}{\mu}+\frac{1}{\rho\omega^2}\frac{\partial^2(\gamma_{k_{\rm s}}-\gamma_{k_{\rm p}})}{\partial{\tilde x}_1^2}&\frac{1}{\rho\omega^2}\frac{\partial^2(\gamma_{k_{\rm s}}-\gamma_{k_{\rm p}})}{\partial \tilde x_1\partial \tilde x_2}\\[4pt]
	\frac{1}{\rho\omega^2}\frac{\partial^2(\gamma_{k_{\rm s}}-\gamma_{k_{\rm p}})}{\partial \tilde x_2\partial \tilde x_1}&\frac{\gamma_{k_{\rm s}}}{\mu}+\frac{1}{\rho\omega^2}\frac{\partial^2(\gamma_{k_{\rm s}}-\gamma_{k_{\rm p}})}{\partial \tilde x_2\partial \tilde x_2}
	\end{bmatrix}.
	\end{equation*}
	We define $(s_1,s_2)^\top:=\mathcal S(\bm\lambda)$. In view of definition \eqref{SDLR}, combining the dual norm estimate with the trace theorem gives
	\begin{equation}
	\label{s_i}
	\begin{aligned}
	|s_i(\tilde{\bm x})|=|\langle G_{i1},\lambda_1\rangle+\langle G_{i2},\lambda_2\rangle|
	&\leq\sum_{j=1}^{2}\|\lambda_j\|_{\bm H^{-1 / 2}(\Gamma_1)}\|G_{ij}\|_{\bm H^{1 / 2}(\Gamma_1)}\\
	&\leq \sum_{j=1}^{2}\|\lambda_j\|_{\bm H^{-1 / 2}(\Gamma_1)}\|G_{ij}\|_{\bm H^1(\Omega_1)},
	\end{aligned}
	\end{equation}
	where $\langle ,\rangle$ denotes the dual pairings on $H^{1/2}(\Gamma_1)\times H^{-1/2}(\Gamma_1)$. By means of the embedding property and Lemma \eqref{FSE}, we can obtain that
	\begin{equation*}
	\begin{aligned}
	\|G_{11}(\tilde{\bm x},:)&\|_{H^1(\Omega_1)} \leq C|\Omega_1|^{1 / 2}\big(\|G_{11}(\tilde{\bm x}, \cdot)\|_{L^{\infty}(\Omega_1)}\\
	&\quad +\|\nabla_{\tilde{\bm{y}}} G_{11}(\tilde{\bm x}, \cdot)\|_{L^{\infty}(\Omega_1)})\leq C(\|\gamma_{k_{\rm s}}\|_{L^{\infty}(\Omega_1)}+\|\nabla_{\tilde{\bm y}}\gamma_{k_{\rm s}}\|_{L^{\infty}(\Omega_1)}\\
	&\quad +\big\|\tfrac{\partial^2(\gamma_{k_{\rm s}}-\gamma_{k_{\rm p}})}{\partial{\tilde x}_1^2}\big\|_{L^{\infty}(\Omega_1)}+\big\|\nabla_{\tilde{\bm y}}\tfrac{\partial^2(\gamma_{k_{\rm s}}-\gamma_{k_{\rm p}})}{\partial{\tilde x}_1^2}\big\|_{L^{\infty}(\Omega_1)}\big)\\
	&\leq C\max\{k_{\rm s}^{-1/2},k_{\rm s}^{5/2}\}{\tilde r}^{-1/2} \leq C\max\{k_{\rm s}^{-1/2},k_{\rm s}^{5/2}\}\{{\operatorname{dist}}(\tilde{\bm x};\Gamma_1)\}^{-1/2}.
	\end{aligned}
	\end{equation*}
	Similar arguments can lead to
	\begin{equation*}
	\begin{aligned}
	&\|G_{ij}(\tilde{\bm x},:)\|_{H^1(\Omega_1)}\leq  C\max\{k_{\rm s}^{-1/2},k_{\rm s}^{5/2}   \}\{{\operatorname{dist}}(\tilde{\bm x};\Gamma_1)\}^{-1/2},\quad i,j=1,2.
	\end{aligned}
	\end{equation*}
	Therefore, we can obtain that
	\begin{equation}
	|s_i(\tilde{\bm x})|\leq C\max\{k_{\rm s}^{-1/2},k_{\rm s}^{5/2}\}\{{\operatorname{dist}}(\tilde{\bm x};\Gamma_1)\}^{-1/2}\|\bm\lambda\|_{\bm H^{-1/2}(\Gamma_1)}.
	\end{equation}
	Thus, we have
	\begin{equation}
	\label{SDLL2}
	\begin{aligned}
	&|\widetilde\Gamma_2|^{-1/2}\|s_i\|_{L^2(\widetilde\Gamma_2)}\leq C\|s_i\|_{L^\infty(\widetilde\Gamma_2)}\\
	&\leq C\max\{k_{\rm s}^{-1/2},k_{\rm s}^{5/2}   \}\sup_{\tilde{\bm x}\in\widetilde\Gamma_2}\{{\operatorname{dist}}(\tilde{\bm x};\Gamma_1)  \}^{-1/2}\|\bm\lambda\|_{\bm H^{-1/2}(\Gamma_1)}\\
	&=C\max\{k_{\rm s}^{-1/2},k_{\rm s}^{5/2}   \}\{{\operatorname{dist}}(\widetilde\Gamma_2;\Gamma_1)  \}^{-1/2}\|\bm\lambda\|_{\bm H^{-1/2}(\Gamma_1)}.
	\end{aligned}
	\end{equation}
	Next, we estimate $|s_i|_{H^{1/2}(\widetilde\Gamma_2)}$. For any $\tilde{\bm x},{\tilde{\bm x}}'\in\widetilde\Gamma_2$,  the mean value theorem gives that
	\begin{equation*}
	|s_i(\tilde{\bm x})-s_i({\tilde{\bm x}}')|\leq C\|\nabla_{\tilde{\bm x}}s_i(\tilde{\bm x})\|_{L^\infty(\widetilde\Gamma_2)}|\tilde{\bm x}-{\tilde{\bm x}}'|.
	\end{equation*}
	Applying definition \eqref{SDLR} again, and invoking the dual norm estimate and trace theorem, we obtain
	\begin{align}
	\nonumber
	|\nabla_{\tilde{\bm x}}s_i(\tilde{\bm x})|&\leq (\|\lambda_1\|_{H^{-1/2}(\Gamma_1)}\|\nabla_{\tilde{\bm x}}G_{i1}\|_{H^{1/2}(\Gamma_1)}+\|\lambda_2\|_{H^{-1/2}(\Gamma_1)}\|\nabla_{\tilde{\bm x}}G_{i2}\|_{H^{1/2}(\Gamma_1)})\\
	\label{DSLDP1}
	&\leq (\|\lambda_1\|_{H^{-1/2}(\Gamma_1)}\|\nabla_{\tilde{\bm x}}G_{i1}\|_{H^1(\Gamma_1)}+\|\lambda_2\|_{H^{-1/2}(\Gamma_1)}\|\nabla_{\tilde{\bm x}}G_{i2}\|_{H^1(\Gamma_1)}).
	\end{align}
	According to the embedding property and Lemma \eqref{GErr}, we have
	\begin{equation}
	\label{DSLDP2}
	\begin{aligned}
	&\| \nabla_{\tilde{\bm x}} G_{i1} \|_{H^1(\Omega_1)} \leq C |\Omega_1|^{1/2} ( \| \nabla_{\tilde{\bm x}} G_{i1} \|_{L^\infty(\Omega_1)}  \\
	&  + \| \nabla_{\tilde{\bm{y}}} \nabla_{\tilde{\bm x}} G_{i1} \|_{L^\infty(\Omega_1)} ) \leq C \max \{ k_{\rm s}, k_{\rm s}^{7/2} \} \{ \operatorname{dist}(\tilde{\bm x}; \Gamma_1) \}^{-1/2}.
	\end{aligned}
	\end{equation}
	Combining \eqref{H213} and \eqref{DSLDP1}-\eqref{DSLDP2} yields
	\begin{equation}
	\label{SDLDPErr2}
	\begin{aligned}
	|\widetilde\Gamma_2|^{-1}|s_i|_{H^{1/2}(\widetilde\Gamma_2)}&\leq C\|\nabla_{\tilde{\bm x}}s_i\|_{L^\infty(\widetilde\Gamma_2)}\\
	&\leq C\max\{k_{\rm s},k_{\rm s}^{7/2}\}\sup_{\tilde{\bm x}\in\widetilde\Gamma_2}\{{\operatorname{dist}}(\tilde{\bm x};\Gamma_1)\}^{-1/2}\|\bm\lambda\|_{H^{-1/2}(\Gamma_1)}\\
	&\leq C\max\{k_{\rm s},k_{\rm s}^{7/2}\}\{{\operatorname{dist}}(\widetilde\Gamma_2;\Gamma_1)\}^{-1/2}\|\bm\lambda\|_{H^{-1/2}(\Gamma_1)}.
	\end{aligned}
	\end{equation}
	Consequently, according to \eqref{SDLL2} and \eqref{SDLDPErr2}, we get
	\begin{equation*}
	\|\mathcal S(\bm\lambda)(\tilde{\bm x})\|_{\bm H^{1 / 2}(\Gamma)} \leq C \max\{k_{\rm s}^{-1/2},k_{\rm s}^{7/2}\}\{\operatorname{dist}(\widetilde\Gamma_2 ; \Gamma_1)\}^{-\frac{1}{2}}\|\bm \lambda\|_{\bm H^{-1 / 2}(\Gamma_1)}.
	\end{equation*}
	This completes the proof.
\end{proof}

\begin{lemma}
	\label{DLPE}
	Under the same conditions as in Lemma \ref{SLPE}, we have that for any $\bm \xi\in \bm H^{1/2}(\Gamma_1),$
	\begin{equation}
	\|\mathcal D(\bm\xi)(\tilde{\bm x})\|_{\bm H^{1 / 2}(\widetilde\Gamma_2)} \leq C \max\{k_{\rm p},k_{\rm s}^{7/2}\}\{\operatorname{dist}(\widetilde\Gamma_2 ; \Gamma_1)\}^{-\frac{1}{2}}\|\bm \xi\|_{\bm H^{1 / 2}(\Gamma_1)},
	\end{equation}
	where $C$ is a positive constant independent of $k_{\rm t}$ and $\bm\xi$.
\end{lemma}
\begin{proof}
	It can be seen from \cite{HW08,BXY17,BXY191} that
	\begin{equation}
	\label{TyE}
	\begin{aligned}
	&\begin{bmatrix}
	Q_{11}&Q_{12}\\[4pt]
	Q_{21}&Q_{22}
	\end{bmatrix}:=\bm T(\partial_{\tilde{\bm y}}, \bm \nu_{\tilde{\bm y}})\mathbb E(\tilde{\bm x},\tilde{\bm y}) =-\bm \nu_{\tilde{\bm y}} \nabla_{\tilde{\bm y}}^{\top}[\gamma_{k_{\rm s}}(\tilde{\bm x},\tilde{\bm y})-\gamma_{k_{\rm p}}(\tilde{\bm x},\tilde{\bm y})] \\[4pt]
	&+\partial_{\bm \nu_{\tilde{\bm y}}} \gamma_{k_{\rm s}}(\tilde{\bm x},\tilde{\bm y}) \mathbb I +M_{\tilde{\bm y}}[2 \mu\mathbb E(\tilde{\bm x},\tilde{\bm y})-\gamma_{k_{\rm s}}(\tilde{\bm x},\tilde{\bm y}) \mathbb I],
	\end{aligned}
	\end{equation}
	where the operator $M$ denotes the G$\ddot{\text u}$nter derivatives matrix given by
	\begin{equation*}
	M_{\tilde{\bm y}}=\begin{bmatrix}
	0&\nu_{\tilde y_2}\partial_{\tilde y_2}-\nu_{\tilde y_1}\partial_{\tilde y_1}\\[4pt]
	\nu_{\tilde y_1}\partial_{\tilde y_1}-\nu_{\tilde y_2}\partial_{\tilde y_2}&0
	\end{bmatrix}.
	\end{equation*}
	Now, we define $(p_1,p_2)^\top=\mathcal D(\bm\xi)$. Using the formula \eqref{DPL}, \eqref{TyE} and Lemma \eqref{FSE}, we can obtain that
	\begin{equation*}
	\begin{aligned}
	&|p_i(\tilde{\bm x})|=\sum_{j=1}^{2}\big|\langle Q_{ji},\xi_j\rangle\big|\leq C\sum_{j=1}^{2}\|Q_{ji}(\tilde{\bm x},:)\|_{L^\infty(\Gamma_1)}\|\xi_j\|_{L^1(\Gamma_1)}\\
	&\leq C\max\{k_{\rm p},k_{\rm s}^{5/2}\}\{\text{dist}(\tilde{\bm x};\Gamma_1)\}^{-1/2}\|\bm\xi\|_{\bm L^1(\Gamma_1)}.
	\end{aligned}
	\end{equation*}
	Since $\widetilde\Omega_{\rm PML}$ is bounded, using the embedding theorem, we arrive at
	\begin{equation*}
	\|\bm\xi\|_{\bm L^1(\Gamma_1)}\leq C\|\bm\xi\|_{\bm L^2(\Gamma_1)}\leq C\|\bm\xi\|_{\bm H^{1/2}(\Gamma_1)}.
	\end{equation*}
	Then we have 
	\begin{equation}
	\label{DLP1}
	\begin{aligned}
	&|\widetilde\Gamma_2|^{-1/2}\|p_i\|_{L^2(\widetilde\Gamma_2)}\leq \|p_i\|_{L^\infty(\widetilde\Gamma_2)}\\
	&\leq \max\{k_{\rm s},k_{\rm s}^{5/2}\}\sup_{\tilde{\bm x}\in\widetilde\Gamma_2}\{\text{dist}(\tilde{\bm x};\Gamma_1)\}^{-1/2}\|\bm\xi\|_{\bm H^{1/2}(\Gamma_1)}\\
	&\leq \max\{k_{\rm s},k_{\rm s}^{5/2}\}\{\text{dist}(\widetilde\Gamma_2;\Gamma_1)\}^{-1/2}\|\bm\xi\|_{\bm H^{1/2}(\Gamma_1)}.
	\end{aligned}
	\end{equation}
	By means of the mean value theorem, we conclude that for any $\tilde{\bm x},{\tilde{\bm x}}'\in\widetilde\Gamma_2$
	\begin{equation*}
	|p_i(\tilde{\bm x})-p_i({\tilde{\bm x}}')|\leq C |\nabla_{\tilde{\bm x}}p_i(\tilde{\bm x})|_{L^\infty(\widetilde\Gamma_2)}|\tilde{\bm x}-{\tilde{\bm x}}'|.
	\end{equation*}
	Utilizing the formula \eqref{TyE} and Lemma \eqref{FSE}, we can obtain that
	\begin{equation*}
	\begin{aligned}
	&|\nabla_{\tilde{\bm x}}p_i(\tilde{\bm x})|\leq \sum_{j=1}^{2}\|\nabla_{\tilde{\bm x}}Q_{ji}(\tilde{\bm x},:)\|_{L^\infty(\Gamma_1)}\|\xi_j\|_{L^1(\Gamma_1)}\\
	&\leq C\max\{k_{\rm p}^{3/2},k_{\rm s}^{7/2}\}\{\text{dist}(\widetilde\Gamma_2;\Gamma_1)\}^{-1/2}\|\bm\xi\|_{\bm H^{1/2}(\Gamma_1)}.
	\end{aligned}
	\end{equation*}
	Therefore, by the definition \eqref{H213}, we can obtain that
	\begin{equation}
	\label{DLP2}
	\begin{aligned}
	&|\widetilde\Gamma_2|^{-1}|p_i(\tilde{\bm x})|_{H^{1/2}(\widetilde\Gamma_2)}\leq C\|\nabla_{\tilde{\bm x}}p_i(\tilde{\bm x})\|_{L^\infty(\widetilde\Gamma_2)}\\
	&\leq C\max\{k_{\rm p}^{3/2},k_{\rm s}^{7/2}\}\{\text{dist}(\widetilde\Gamma_2;\Gamma_1)\}^{-1/2}\|\bm\xi\|_{\bm H^{1/2}(\Gamma_1)}.
	\end{aligned}
	\end{equation}
	A combination of \eqref{DLP1} and \eqref{DLP2} completes the proof.
\end{proof}

With the above preparations, we now turn to the convergence analysis of the rectangular RCL. 

\begin{theorem}
	\label{UE}
	
	If $-\rho\omega^2$ is not an eigenvalue of the interior Dirichlet problem for the Lam\'e operator $\Delta^*$, the equation \eqref{PMLeq2} has a unique solution $\hat{\bm u} \in H^1(\Omega_b)$, which converges to the solution $\bm u$ of the original problem \eqref{OP} exponentially in $\Omega_1$,
	\begin{equation}
	\label{FError}
	\|\bm u-\hat{\bm u}\|_{\bm H^1(\Omega_1) \mid}\leq C\max\{k_{\rm p},k_{\rm s}^{-1/2},k_{\rm s}^{7/2}\}e^{-\frac{1}{2}\tau_0\operatorname{dist}(\Gamma_b,\Gamma_1)}\|{\bm u}\|_{\bm H^{1/2}(\Gamma_1)},
	\end{equation}
	where $C$ is a positive constant independent of $k_{\rm t}$.
\end{theorem}

\begin{proof}
	Mapping \eqref{PMLeq2} back to the $(\tilde x_1,\tilde x_2)$ coordinates through \eqref{RCL}, we can obtain that
	\begin{equation}
	\label{OCEq1}
	\begin{dcases}
	\Delta_{\tilde{\bm x}}^*\widehat{\bm U}(\tilde{\bm x})+\rho\omega^2\theta(\tilde{\bm x})\widehat{\bm U}(\tilde{\bm x})=\bm f(\bm x), \quad& \tilde{\bm x}\in \widetilde\Omega_b, \\
	\widehat{\bm U}(\tilde{\bm x})=\bm g(\tilde{\bm x}), \quad& \tilde{\bm x}\in \Gamma, \\
	\widehat{\bm U}(\tilde{\bm x})=0,\quad& \tilde{\bm x}\in\widetilde\Gamma_2,
	\end{dcases}
	\end{equation}
	where $\widehat{\bm U}(\tau, \theta)=\hat{\bm u}(r, \theta)$. Assuming that $-\rho\omega^2$ is not an eigenvalue of the interior Dirichlet problem for the Lam\'e operator $\Delta^*$, it can be seen from \cite{KGBB79} that \eqref{OCEq1} has a unique solution $\widehat{\bm U}$ such that
	\begin{equation}
	\label{IDUE}
	\|\widehat{\bm U}\|_{\bm H^1(\widetilde\Omega_b)}\leq C(\|\bm g\|_{\bm H^{1/2}(\Gamma)}+\|\bm f\|_{H^1(\Omega_1)'}),
	\end{equation}
	where the positive constant $C$ depends on $k_{\rm t}$ and $\widetilde\Omega_b$.
	
	From \eqref{SDLR}, we know that the solution of the original scattering problem \eqref{PMLOeq} on $\widetilde\Gamma_2$ can be represented as 
	\begin{equation}
	\label{SDU2}
	\bm q(\tilde{\bm x})=\mathcal D(\bm\xi)(\tilde{\bm x})-\mathcal S(\bm\lambda)(\tilde{\bm x}),\quad\tilde{\bm x}\in \widetilde\Gamma_2,
	\end{equation}
	where
	\begin{equation*}
	\bm \xi:=\widetilde{\bm U}|_{\Gamma_1}=\tilde{\bm u}|_{\Gamma_1}={\bm u}|_{\Gamma_1},\quad \bm\lambda=\bm T(\partial,\bm\nu)\bm \xi.
	\end{equation*}
	Then the original problem \eqref{PMLOeq} enclosed by $\widetilde\Gamma_2$ becomes 
	\begin{equation}
	\label{OCEq2}
	\begin{dcases}
	\Delta_{\tilde{\bm x}}^*\widetilde{\bm U}(\tilde{\bm x})+\rho\omega^2\theta(\tilde{\bm x})\widetilde{\bm U}(\tilde{\bm x})=\bm f(\bm x), \quad  & \tilde{\bm x}\in \widetilde\Omega_b, \\
	\widetilde{\bm U}(\tilde{\bm x})=\bm g(\tilde{\bm x}), \quad  & \tilde{\bm x}\in \Gamma, \\
	\widetilde{\bm U}(\tilde{\bm x})=\bm q(\tilde{\bm x}),\quad & \tilde{\bm x}\in\widetilde\Gamma_2.
	\end{dcases}
	\end{equation}
	Introducing ${\bm e}=\widetilde{\bm U}-\widehat{\bm U}$ and subtracting \eqref{OCEq1} from \eqref{OCEq2} yields the following error equation
	\begin{equation}
	\label{ErrorEq1}
	\begin{dcases}
	\Delta_{\tilde{\bm x}}^*{\bm e}(\tilde{\bm x})+\rho\omega^2\theta(\tilde{\bm x}){\bm e}(\tilde{\bm x})=\bm 0, \quad &\tilde{\bm x}\in \widetilde\Omega_b, \\
	{\bm e}(\tilde{\bm x})=\bm 0, \quad &\tilde{\bm x}\in \Gamma, \\
	{\bm e}(\tilde{\bm x})=\bm q(\tilde{\bm x}),\quad &\tilde{\bm x}\in\widetilde\Gamma_2.
	\end{dcases}
	\end{equation}
	Consequently, the regularity result \eqref{IDUE} implies
	\begin{equation}
	\label{ErrE1}
	\|{\bm e}\|_{H^1(\widetilde\Omega_b)}\leq C\|\bm q\|_{H^{1/2}(\widetilde\Gamma_2)}.
	\end{equation}
	Therefore, from Lemma \eqref{SLPE}, Lemma \eqref{DLPE} and \eqref{DtNError}, we have
	\begin{equation}
	\label{FError2}
	\|\bm q\|_{H^{1/2}(\widetilde\Gamma_2)}\leq C\max\{k_{\rm p},k_{\rm s}^{-1/2},k_{\rm s}^{7/2}\}\{{\operatorname{dist}}(\widetilde\Gamma_2;\Gamma)\}^{-1/2}\|\bm \xi\|_{\bm H^{1/2}(\Gamma_1)},
	\end{equation}
	when $k_{\rm t}\operatorname{dist}(\widetilde\Gamma_2;\Gamma_1)>1$. Applying the definition of the mapping $\tau(r,\theta)$ in \eqref{RCL} gives
	\begin{equation}
	\label{distE}
	\operatorname{dist}(\widetilde\Gamma_2;\Gamma_1)=a(\theta)e^{\tau_0\operatorname{dist}(\Gamma_2,\Gamma_1)}.
	\end{equation}
	In view of \eqref{ErrE1}-\eqref{distE}, we can obtain that
	\begin{equation}
	\label{FError}
	\|\bm e\|_{H^1(\tilde\Omega_b)}\leq C\max\{k_{\rm p},k_{\rm s}^{-1/2},k_{\rm s}^{7/2}\}e^{-\frac{1}{2}\tau_0\operatorname{dist}(\Gamma_b,\Gamma)}\|\widetilde{\bm U}\|_{\bm H^{1/2}(\Gamma_1)}.
	\end{equation}
	Since $\tau(r,\theta)$ reduces to the identity mapping $\Omega_1$, we have $\bm u=\tilde{\bm u}=\widetilde{\bm U}$ and $\hat{\bm u}=\widehat{\bm U}$ in $\Omega_1$. In view of \eqref{ErrE1}-\eqref{distE}, we conclude that
	\begin{equation}
	\label{FError1}
	\begin{aligned}
	&\|\bm u-\hat{\bm u}\|_{H^1(\Omega_1)}=\|\widetilde{\bm U}-\widehat{\bm U}\|_{\bm H^1(\Omega_1)}\leq\|{\bm e}\|_{\bm H^1(\widetilde\Omega_b)}\\
	&\leq C\max\{k_{\rm p},k_{\rm s}^{-1/2},k_{\rm s}^{7/2}\}e^{-\frac{1}{2}\tau_0\operatorname{dist}(\Gamma_b,\Gamma)}\|{\bm u}\|_{\bm H^{1/2}(\Gamma_1)}.
	\end{aligned}
	\end{equation}
	This ends the proof.
\end{proof}

\section{Numerical implementation}
\label{sec:4}

In this section, we derive a weak formulation of the coupled RCL problem \eqref{OP2} and develop its high-order spectral element discretization. The weak formulation is established in terms of the displacement field in the physical domain $\Omega_1$ and the transformed compressional and shear potentials in the RCL domain $\Omega_2$, coupled through the interface conditions on $\Gamma_1$.

We define the space
\begin{equation*}
\bm H_{0, \Gamma}^1(\Omega_1)=\{\bm u \in \bm H^1(\Omega_1):\left.\bm u\right|_{\Gamma}=\bm 0\},\quad\bm Y=\bm H_{0, \Gamma}^1(\Omega_1)\times H^1(\Omega_2)\times H^1(\Omega_2).
\end{equation*}
The weak formulation of \eqref{OP2} is to find $(\hat{\bm u},\hat\phi_{\rm p},\hat \phi_{\rm s}) \in\bm Y$ such that 

\begin{equation}
\label{WF}
\begin{aligned}
&\mathcal A_{\Omega_1}(\hat{\bm u},\hat{\bm\varphi})+\mathcal A_{\Gamma_1}(\hat\phi_p,\hat\phi_s,\hat{\bm\varphi})=-(\bm f,\hat{\bm\varphi})_{\Omega_1}, \\[0.2em]
&\mathcal A_{\Omega_2}(\hat\phi_{\rm p}, \hat\phi_{\rm s} ;\hat\psi_{\rm p}, \hat\psi_{\rm s})+\mathcal A_{\Gamma_2}(\hat{\bm u},\hat\phi_p,\hat\phi_s,\hat\psi_p,\hat\psi_s)=0,\quad \quad \forall(\hat{\bm\varphi},\hat\psi_{\rm p},\hat\psi_{\rm s}) \in \bm Y,
\end{aligned}
\end{equation}
where
\begin{align*}
&\mathcal A_{\Omega_1}(\hat{\bm u},\hat{\bm\varphi})=\tfrac{\mu}{2} (\nabla \hat{\bm{u}}+\nabla \hat{\bm{u}}^\top,\nabla\hat{\bm\varphi}+\nabla\hat{\bm\varphi}^\top)_{\Omega_1}+\lambda(\nabla \cdot \hat{\bm{u}},\nabla \cdot \hat{\bm\varphi})_{\Omega_1}-\rho\omega^2(\theta\hat{\bm u},\hat{\bm\varphi})_{\Omega_1}\\[0.2em]
&\mathcal A_{\Gamma_1}(\hat\phi_p,\hat\phi_s,\hat{\bm\varphi})=\\
&\lambda k_{\rm p}^2\langle\hat\phi_{\rm p}\bm\nu,\hat{\bm\varphi}\rangle_{\Gamma_1}+\mu k_{\rm s}^2\langle\hat\phi_{\rm s}\bm\tau,\hat{\bm\varphi}\rangle_{\Gamma_1}+2\mu(\langle\nabla \hat\phi_{\rm p},\partial_{\bm\nu}\hat{\bm\varphi}\rangle_{\Gamma_1}+\langle\nabla\hat\phi_{\rm s},\partial_{\bm\tau}\hat{\bm\varphi}\rangle_{\Gamma_1})\\[0.2em]
&\mathcal A_{\Omega_2}(\hat\phi_{\rm p}, \hat\phi_{\rm s} ;\hat\psi_{\rm p}, \hat\psi_{\rm s})=\\&(\mathbb A\nabla\hat\phi_{\rm p}, \nabla(\tfrac{r}{\tau}\hat\psi_{\rm p}))_{\Omega_2}+(\mathbb A\nabla\hat\phi_{\rm s}, \nabla(\tfrac{r}{\tau}\hat\psi_{\rm s}))_{\Omega_2} -k_{\rm p}^2(\tau_r\hat\phi_{\rm p}, \hat\psi_{\rm p})_{\Omega_2}-k_{\rm s}^2(\tau_r\hat\phi_{\rm s}, \hat\psi_{\rm s})_{\Omega_2},\\[0.2em]
&\mathcal A_{\Gamma_2}(\hat{\bm u},\hat\phi_p,\hat\phi_s,\hat\psi_p,\hat\psi_s)=-\langle\partial_\tau \hat\phi_{\rm s},\hat\psi_{\rm p}\rangle_{\Gamma_1}+\langle\partial_\tau\hat\phi_{\rm p},\hat\psi_{\rm s}\rangle_{\Gamma_1}+\langle \bm\nu\cdot\hat{\bm u},\hat\psi_{\rm p}\rangle_{\Gamma_1}\\
&-\langle \bm\tau\cdot\hat{\bm u},\hat\psi_{\rm s}\rangle_{\Gamma_1}+\langle\tilde\partial_\tau \hat\phi_{\rm s},\tfrac{r}{\tau}\hat\psi_{\rm p}\rangle_{\Gamma_2}-\langle\tilde\partial_\tau\hat\phi_{\rm p},\tfrac{r}{\tau}\hat\psi_{\rm s}\rangle_{\Gamma_2}
\end{align*}
Here, for scalar fields $p$ and $q$, vector fields $\bm P$ and $\bm Q$, and matrices $\mathbb P$ and $\mathbb Q$, $()_{\Omega_i}$ denotes the inner product as
\begin{equation*}(\mathbb P,\mathbb Q)_{\Omega_1}=\int_{\Omega_1}\mathbb P:\overline{\mathbb Q}d\bm x, \quad (\bm P,\bm Q)_{\Omega_1}=\int_{\Omega_1}\bm P\cdot\overline{\bm Q}d\bm x,\quad (p,q)_{\Omega_2}=\int_{\Omega_2}p\bar{q}d\bm x,
\end{equation*}
and $\langle \cdot, \cdot \rangle$ represents the duality pairing
\begin{equation*}
\left\langle \bm P,\bm Q\right\rangle_{\Gamma}=\int_\Gamma \bm P\cdot\overline{\bm Q} ds,\quad \left\langle p,q\right\rangle_{\Gamma}=\int_\Gamma p\overline q ds.
\end{equation*}

As shown in \eqref{psR2}, the fields $\hat\phi_{\rm t}$, $\rm t=p,s$ are highly oscillatory near $\Gamma_1$. To eliminate oscillations, we introduce the substitution $\hat\phi_{\rm t}=w_{\rm t}\hat v_{\rm t}$ with
\begin{equation}
\label{WR}
w_{\rm t}(r):=\begin{cases}
1, &r\in\Omega_1,\\
e^{{\rm i}k_{\rm t}(\tau(r,\theta)-a(\theta))},&r\in\Omega_2,
\end{cases}
\end{equation}
where $\hat{v}_{\rm t}$ are free of oscillation. Then we present a detailed expansion of the transformed sesquilinear form. Specifically, the problem \eqref{WF} is reformulated as follows: find $(\hat{\bm u},\hat v_{\rm p}, \hat v_{\rm s}) \in\bm Y$ such that:
\begin{equation}
\label{WF1}
\begin{aligned}
&\mathcal A_{\Omega_1}(\hat{\bm u},\hat{\bm\varphi})+\mathcal A_{\Gamma_1}(w_p\hat\phi_p,w_s\hat\phi_s,\hat{\bm\varphi})=-(\bm f,\hat{\bm\varphi})_{\Omega_1}, \\[0.2em]
&\breve{\mathcal A}_{\Omega_2}(\hat v_{\rm p}, \hat v_{\rm s};\hat\zeta_{\rm p},\hat\zeta_{\rm s})+\mathcal A_{\Gamma_2}(\hat{\bm u},w_p\hat v_p,w_s\hat v_s,\tfrac{r}{\tau}w_p\hat\zeta_p,\tfrac{r}{\tau}w_s\hat\zeta_s)=0
\end{aligned}
\end{equation}
for all $(\hat{\bm\varphi},\hat\zeta_{\rm p},\hat\zeta_{\rm s}) \in \bm Y$, where
\begin{equation}
\label{A2}
\begin{aligned}
&\breve{\mathcal A}_{\Omega_2}(\hat v_{\rm p}, \hat v_{\rm s};\hat\zeta_{\rm p},\hat\zeta_{\rm s})=(\mathbb A\nabla(w_p\hat v_{\rm p}), \nabla (\tfrac{r}{\tau}w_p\hat\zeta_{\rm p}))_{\Omega_2}\\
&+(\mathbb A\nabla(w_s\hat v_{\rm s}), \nabla (\tfrac{r}{\tau}w_s\hat\zeta_{\rm s}))_{\Omega_2} -k_{\rm p}^2(\tau_r\hat v_{\rm p}, \hat\zeta_{\rm p})_{\Omega_2}-k_{\rm s}^2(\tau_r\hat v_{\rm s}, \hat\zeta_{\rm s})_{\Omega_2}.
\end{aligned}
\end{equation}

\begin{theorem}
	\label{Th4.1}
	The sesquilinear form $\breve{\mathcal A}_{\Omega_2}(\hat v_{\rm p}, \hat v_{\rm s};\hat\zeta_{\rm p},\hat\zeta_{\rm s})$ in \eqref{A2} can be rewritten as
	\begin{equation}
	\breve{\mathcal A}_{\Omega_2}(\hat v_{\rm p}, \hat v_{\rm s};\hat\zeta_{\rm p},\hat\zeta_{\rm s})=\breve{\mathcal A}^{\rm p}_{\Omega_2}(\hat v_{\rm p};\hat\zeta_{\rm p})+\breve{\mathcal A}^{\rm s}_{\Omega_2}(\hat v_{\rm s};\hat\zeta_{\rm s}),
	\end{equation}
	where
	\begin{equation}
	\label{At2}
	\breve{\mathcal A}^{\rm t}_{\Omega_2}(\hat v_{\rm p};\hat\zeta_{\rm p})=(\tfrac{r}{\tau}\mathbb A{\nabla} \hat v_{\rm t}, {\nabla} \hat\zeta_{\rm t})+(\mathbb R_\theta\bm{p} \cdot {\nabla}\hat v_{\rm t}, \hat\zeta_{\rm t})+(\hat v_{\rm t}, \mathbb R_\theta\bm{q}^* \cdot {\nabla} \hat\zeta_{\rm t})+(\breve{n} \hat v_{\rm t}, \hat\zeta_{\rm t}),
	\end{equation}
	with 
	\begin{equation}
	\begin{aligned}
	& \bm{p}:=w_t \mathbb B \widehat{\nabla}(\tfrac{rw_{\rm t}^*}{\tau} ), \quad \bm{q}:=\tfrac{rw_t^*}{\tau} \mathbb B \widehat{\nabla} w_t, \quad \breve{n}:=(\widehat{\nabla}w_{\rm t})^\top \mathbb B \widehat{\nabla}(\tfrac{rw_{\rm t}^*}{\tau} )-k_{\rm t}^2\tau_r.
	\end{aligned}
	\end{equation}
	Furthermore, in $\Omega_2$, the terms $\mathbb B$, $\bm p$, $\bm q$ and $\breve n$ can be evaluated by
	\begin{equation*}
	\mathbb B=\begin{bmatrix}
	\tfrac{\tau}{r\tau_r}+\tfrac{\tau_\theta^2}{r\tau_r\tau}&-\tfrac{\tau_\theta}{\tau}\\[4pt]
	-\tfrac{\tau_\theta}{\tau}&\tfrac{r\tau_r}{\tau}.
	\end{bmatrix}
	\end{equation*}
	The terms $\bm p$, $\bm q$ and $\breve n$ can be evaluated by
	\begin{equation*}
	\begin{aligned}
	&p_1=\tfrac{1}{r\tau_r}-\tfrac{1}{\tau}-{\rm i}k_{\rm t}-\tfrac{{\rm i}k_{\rm t}\tau_\theta a_\theta}{\tau^2}+\tfrac{\tau_\theta^2}{r\tau_r\tau^2},\quad p_2=-\tfrac{\tau_\theta}{\tau^2}+\tfrac{{\rm i}k_{\rm t}r\tau_ra_\theta}{\tau^2},\\
	&q_1={\rm i}k_{\rm t}+\tfrac{{\rm i}k_{\rm t}a_\theta\tau_\theta}{\tau^2},\quad q_2=-\tfrac{{\rm i}k_{\rm t}r\tau_ra_\theta}{\tau^2},\quad \breve n={\rm i}k_{\rm t}(\tfrac{1}{r}-\tfrac{\tau_r}{\tau}+\tfrac{\tau_\theta a_\theta}{\tau^2 r}-\tfrac{{\rm i}k_{\rm t}\tau_r a_\theta^2}{\tau^2}).
	\end{aligned}
	\end{equation*}
	
\end{theorem}
\begin{proof}
	
	According to \eqref{Hna}, we have
	\begin{equation*}
	\begin{aligned}
	&(\mathbb A\nabla(w_{\rm t}\hat v_{\rm t}))^\top \nabla(\tfrac{rw^*_{\rm t}}{\tau}\hat \zeta^*_{\rm t})=(w_t\widehat\nabla \hat v_t+\hat v_t\widehat\nabla w_t)^\top\mathbb B(\tfrac{rw^*_t}{\tau}\widehat\nabla \hat \zeta^*_t+\hat \zeta^*_t\widehat\nabla \tfrac{rw^*_t}{\tau})=(\tfrac{r}{\tau}\mathbb B\widehat\nabla \hat v_t)^\top\widehat\nabla \hat \zeta^*_t\\[0.2em]
	&+(\widehat\nabla\hat v_t)^\top w_t\mathbb B\widehat\nabla(\tfrac{rw_t^*}{\tau})\hat \zeta^*_t+\hat v_t(\tfrac{rw_t^*}{\tau}\mathbb B\widehat \nabla w_t)^\top\widehat\nabla\hat\zeta_t^*+\hat v_t(\widehat\nabla w_t)^\top\mathbb B\widehat\nabla(\tfrac{r}{\tau} w_t^*)\hat\zeta_t^*\\[0.2em]
	&=(\mathbb B\widehat\nabla v_t,\widehat\nabla\hat\zeta_t)_{\Omega_2}+(\bm p\cdot\widehat\nabla v_t,\hat\zeta_t)_{\Omega_2}+( v_t,\bm q^*\cdot\widehat\nabla\hat\zeta_t)_{\Omega_2}+((\widehat\nabla w_t)^\top\mathbb B\widehat\nabla\tfrac{r w^*_t}{\tau}\hat v_t,\hat\zeta_t)_{\Omega_2}.
	\end{aligned}
	\end{equation*}
	Using the property $\widehat{\nabla}=\mathbb R_\theta^t \nabla $, we can complete the proof of \eqref{At2}. A direct calculation yields that
	\begin{equation}
	\begin{aligned}
	\label{gradw}
	&\tau_r=\tau_0\tau,\quad \tau_\theta=\Big[\tfrac{a_\theta}{a}-\tau_0a_\theta\Big]\tau,\quad \tfrac{dw_t}{dr}={\rm i}k_{\rm t}\tau_rw_t,\quad \tfrac{dw_t}{d\theta}={\rm i}k_{\rm t}(\tau_\theta-a_\theta)w_t,\\
	&\tfrac{dw^*_t}{dr}=-{\rm i}k_{\rm t}\tau_rw^*_t,\quad \tfrac{dw^*_t}{d\theta}=-{\rm i}k_{\rm t}(\tau_\theta-a_\theta)w^*_t.
	\end{aligned}
	\end{equation}
	Then can obtain that
	\begin{equation*}
	\bm p=w_t\mathbb B\begin{pmatrix}
	\tfrac{r}{\tau}\tfrac{\partial w_t^*}{\partial r}+w_t^*\tfrac{\partial}{\partial r}(\tfrac{r}{\tau})\\[4pt]
	\tfrac{1}{r}(\tfrac{r}{\tau}\tfrac{\partial w_t^*}{\partial \theta}+w_t^*\tfrac{\partial}{\partial \theta}(\tfrac{r}{\tau}))
	\end{pmatrix}=\mathbb B\begin{pmatrix}
	-\tfrac{{\rm i}k_{\rm t}r\tau_r}{\tau}+(\tfrac{1}{\tau}-\tfrac{r\tau_r}{\tau^2})\\[4pt]
	-\tfrac{{\rm i}k_{\rm t}(\tau_\theta-a_\theta)}{\tau}-\tfrac{r\tau_\theta}{\tau^2}
	\end{pmatrix},
	\end{equation*}
	which yields that
	\begin{equation*}
	p_1=\tfrac{1}{r\tau_r}-\tfrac{1}{\tau}-{\rm i}k_{\rm t}-\tfrac{{\rm i}k_{\rm t}\tau_\theta a_\theta}{\tau^2}+\tfrac{\tau_\theta^2}{r\tau_r\tau^2},\quad p_2=-\tfrac{\tau_\theta}{\tau^2}+\tfrac{{\rm i}k_{\rm t}r\tau_ra_\theta}{\tau^2}.
	\end{equation*}
	For the term $\bm q$, we have
	\begin{equation*}
	\bm{q}:=\tfrac{rw_t^*}{\tau} \mathbb B \widehat{\nabla} w_t=\tfrac{r}{\tau}\mathbb B\begin{pmatrix}
	{\rm i}k_{\rm t}\tau_r\\[4pt]
	\tfrac{{\rm i}k_{\rm t}(\tau_\theta-a_\theta)}{r}
	\end{pmatrix}=\begin{pmatrix}
	{\rm i}k_{\rm t}+\tfrac{{\rm i}k_{\rm t}a_\theta\tau_\theta}{\tau^2}\\[4pt]
	-\tfrac{{\rm i}k_{\rm t}r\tau_ra_\theta}{\tau^2}
	\end{pmatrix}
	\end{equation*}
	Finally, we have
	\begin{equation*}
	\breve n=\begin{pmatrix}
	{\rm i}k_{\rm t}\tau_r\\[4pt]
	\tfrac{{\rm i}k_{\rm t}(\tau_\theta-a_\theta)}{r}
	\end{pmatrix}^\top\mathbb B\begin{pmatrix}
	-\tfrac{{\rm i}k_{\rm t}r\tau_r}{\tau}+(\tfrac{1}{\tau}-\tfrac{r\tau_r}{\tau^2})\\[4pt]
	-\tfrac{{\rm i}k_{\rm t}(\tau_\theta-a_\theta)}{\tau}-\tfrac{r\tau_\theta}{\tau^2}
	\end{pmatrix}-k_{\rm t}^2\tau_r={\rm i}k_{\rm t}\Big(\tfrac{1}{r}-\tfrac{\tau_r}{\tau}+\tfrac{\tau_\theta a_\theta}{\tau^2 r}-\tfrac{{\rm i}k_{\rm t}\tau_r a_\theta^2}{\tau^2}\Big).
	\end{equation*}
	This completes the proof.
\end{proof}

We now introduce Spectral-element discretization based on the variational
formulation in Theorem~\ref{Th4.1}. The physical region $\Omega_1$ and
the RCL layer $\Omega_2$ are partitioned into non-overlapping quadrilateral
elements $\overline{\Omega}_1=\bigcup_{i=1}^{l_1}\overline{\Omega}_1^{(i)}$ and $\overline{\Omega}_2=\bigcup_{i=1}^{l_2}\overline{\Omega}_2^{(i)}$.
For each element, we use the Gordon-Hall mapping $\mathcal T_j^i:\widehat K=[-1,1]^2\to \Omega_j^{(i)}$, $j=1,2$ to transform the reference square $\widehat K$ onto the element $\Omega_j^{(i)}$. Let $\mathbb P_N([-1,1])$ denote the space of polynomials of degree at most
$N$ on $[-1,1]$, and define
\begin{equation*}
\mathbb Q_{M,N}:=\mathbb P_M([-1,1])\times \mathbb P_N([-1,1]).
\end{equation*}
Then we define the finite-dimensional approximation space by
\begin{equation*}
\begin{split}
&H_N=\{(\hat{\bm u}_N,\hat v_{{\rm p},N},\hat v_{{\rm s},N}):
\hat{\bm u}_N\in\bm H^1(\Omega_1),\quad(\bm u_N|_{\Omega_1^{(i)}})\circ\mathcal T_1^i\in \mathbb Q_{N_1,N_1}^2,
\quad i=1,\ldots,l_1, \\
&\hat v_{{\rm t},N}\in H^1(\Omega_2),\quad(\hat v_{{\rm t},N}|_{\Omega_2^{(i)}})\circ\mathcal T_2^i
\in \mathbb Q_{N_1,N},
\quad i=1,\ldots,l_2,\quad {\rm t}={\rm p},{\rm s}\}.
\end{split}
\end{equation*}
The approximate potentials in the layer are then reconstructed by $\phi_{{\rm t},N}=w_{\rm t}v_{{\rm t},N}$. Then the discrete problem is to find $(\bm u_N,\hat v_{{\rm p},N},\hat v_{{\rm s},N})\in H_N$ such that
\begin{equation*}
\mathcal A_N((\hat{\bm u}_N,\hat v_{{\rm p},N},\hat v_{{\rm s},N}),(\hat{\bm\varphi}_N,\hat\zeta_{{\rm p},N},\hat\zeta_{{\rm s},N}))=\mathcal F_N(\hat{\bm\varphi}_N,\hat\zeta_{{\rm p},N},\hat\zeta_{{\rm s},N}), \forall (\hat{\bm\varphi}_N,\hat\zeta_{{\rm p},N},\hat\zeta_{{\rm s},N})\in H_N.
\end{equation*}
Here $\mathcal A_N$ and $\mathcal F_N$ are obtained from the continuous bilinear form in Theorem~\ref{Th4.1} by replacing the continuous functions with their spectral-element approximations.

The coupling across the interface $\Gamma_1$ is imposed through the variational formulation. In particular, the standard transmission conditions, namely the continuity of displacement and traction across $\Gamma_1$, are incorporated in the boundary terms coupling the displacement unknown in $\Omega_1$ with the Helmholtz potentials in $\Omega_2$. After assembling all elemental contributions, the method leads to a finite-dimensional complex linear system for $\hat{\bm u}_N$, $\hat v_{{\rm p},N}$ and $\hat v_{{\rm s},N}$.

\section{Numerical experiments}
\label{sec:5}

This section provides some numerical experiments to verify the effectiveness of the proposed RCL method. Here, we set $N_1=30$, which ensure that the numerical error is primarily dominated by the approximation in the RCL layer. In addition, the maximum relative errors presented in this section are calculated in accordance with the expression
\begin{equation*}
\epsilon_{\infty}:=\frac{\max _{\bm x \in \Omega_1}\left\{|\bm u^{\text {num }}(\bm x)-\bm u^{\text {ex }}(\bm x)|\right\}}{\max _{\bm x \in \Omega_1}\left\{|\bm u^{\text {ex }}(\bm x)|\right\}}.
\end{equation*}
In all examples except Example \ref{Example:5.4}, the exact solution is given by a pressure point source located at a point $\bm z$ within the scatterer, as follows:
\begin{equation*}
\bm u^{\rm ex}=\nabla_{\bm x}H_0^1(k_{\rm p}|\bm x-\bm z|)+\bm\curl_{\bm x} H_0^1(k_{\rm s}|\bm x-\bm z|).
\end{equation*}
The boundaries $\Gamma$, $\Gamma_1$ and $\Gamma_2$ are represented in polar form by $r=R_0(\theta)$, $r=R_1(\theta)$ and $r=R_2(\theta)$ for $\theta\in[0,2\pi)$, respectively.

\subsection{Circular RCL layer}
\label{Example:5.1}

In this subsection, we present a series of numerical examples to demonstrate the high accuracy of the proposed method, especially for high-frequency, low-frequency, large-wavenumber-ratio, and imaginary-wavenumber problems. Consider the elastic scattering by a circular scatterer of radius $0.5$ centered at the origin. Let $\Gamma_1$ and $\Gamma_2$ be $r=1$ and $r=1+d$, respectively.  We fix $\bm z=(0.1,0.2)$. Unless otherwise specified, we set $\rho=2$, $\lambda=3$, $\mu=1$. Then we have $k_{\rm p}=\sqrt{\tfrac{2}{5}}\omega$ and $k_{\rm s}=\sqrt{2}\omega$. Let $l_1=20$ and $l_2=10$, i.e., the computational domain is partitioned into 30 non-overlapping quadrilateral elements.   

In the first example, we test the exponential convergence of the RCL method. We compute relative errors for different values of $d$ and $\tau_0$. As shown in Figures \ref{Exampledtau0}, we can observe that the relative error decays exponentially with increasing RCL thickness or absorption coefficient $\tau_0$, which is consistent with our theoretical proof. 
\begin{figure}[htb]
	\centering
	\begin{tabular}{cc}
		\includegraphics[scale=0.13]{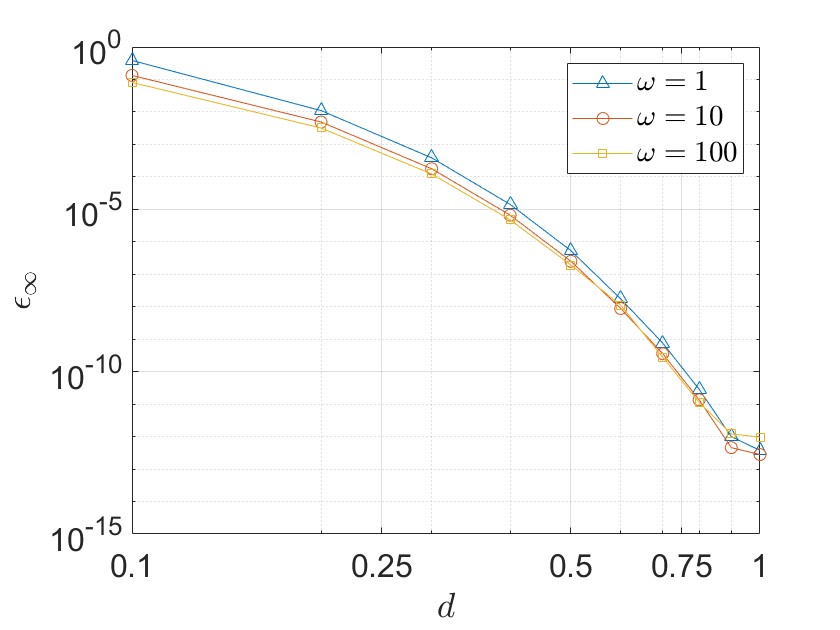} &
		\includegraphics[scale=0.13]{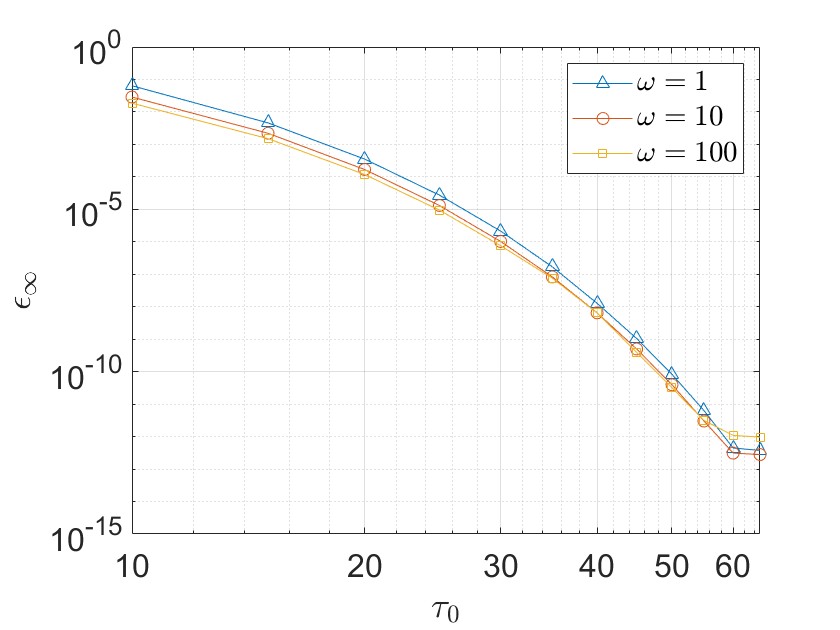} \\
		(a) Errors against $d$  &(b) Errors against $\tau_0$
	\end{tabular}
	\caption{Numerical errors against the thickness $d$ and the parameters $\tau_0$ of the RCL layer. (a) Vary $d$ with a fixed $\tau_0=60$; (b) Vary $\tau_0$ with a fixed $d=0.1$}
	\label{Exampledtau0}
\end{figure}

Next, we test the performance of the RCL method in both low- and high-frequency regimes. Table \ref{Table1} shows that the accuracy for different layer thicknesses in different relatively high- and low-frequency regimes, respectively. We can observe that, with a suitable choice of the parameter $\tau_0$, the accuracy of the RCL method is insensitive to variations in the layer thickness $d$. Consequently, for any given frequency, the precision can be ensured using a constant RCL thickness by adjusting $\tau_0$, in contrast to the $0.5-2$ wavelengths typically required by PML and PAL methods.

\begin{table}
	\centering
	\caption{ Errors versus thickness $d$ for RCL with $N=60$ and $\tau_0=\frac{60}{d}$.}
	\begin{tabular}{lllll}
		\hline$\omega$ & \multicolumn{1}{c}{$d=1$}  & $d=0.1$ & $d=0.01$ & $d=0.001$ \\
		\cline { 2 - 5 }50 &$6.02 \rm{E}-13$  &   $1.36 \rm{E}-12$ & $3.94 \rm{E}-11$ &$2.33 \rm{E}-9$  \\
		100 &$1.13 \rm{E}-12$  &   $1.62 \rm{E}-12$ &$3.59 \rm{E}-11$  &$2.34 \rm{E}-9$  \\
		150 &$1.92 \rm{E}-12$  &  $2.78 \rm{E}-12$ &$4.80 \rm{E}-11$  &$2.13 \rm{E}-9$  \\
		200 &$2.07 \rm{E}-12$  &   $2.05 \rm{E}-12$ &$2.34 \rm{E}-11$  &$1.42 \rm{E}-9$  \\
		250 &$3.42 \rm{E}-12$  &  $2.78 \rm{E}-12$ &$2.73 \rm{E}-11$  &$1.09 \rm{E}-9$  \\
		300 &$6.35 \rm{E}-12$  &   $3.12 \rm{E}-12$ &$2.03 \rm{E}-11$  &$1.35 \rm{E}-9$  \\
		\hline
		$10^{-4}$ &$1.61 \rm{E}-10$  &   $1.67 \rm{E}-10$ &$1.71 \rm{E}-10$  &$4.12 \rm{E}-10$  \\
		$10^{-5}$ &$6.64 \rm{E}-11$  &  $6.56 \rm{E}-11$ &$7.21 \rm{E}-11$  &$6.59 \rm{E}-11$  \\
		$10^{-6}$ &$9.84 \rm{E}-11$  &   $9.15 \rm{E}-11$ &$9.93 \rm{E}-11$  &$9.79 \rm{E}-11$  \\
		$10^{-7}$ &$9.09 \rm{E}-10$  &   $6.03 \rm{E}-10$ &$5.55 \rm{E}-10$  &$1.34 \rm{E}-10$  \\
		$10^{-8}$ &$1.17 \rm{E}-10$  &   $1.55 \rm{E}-10$ &$1.54 \rm{E}-10$  &$1.18 \rm{E}-10$  \\
		\hline
	\end{tabular}
	\label{Table1}
\end{table}

Finally, we consider two challenging cases discussed above: high-contrast Lamé parameters and imaginary frequencies. The purpose is to test the RCL method when the P- and S-wave scales are strongly separated or when slowly decaying evanescent components occur. Figure~\ref{ExampleHB} shows that the method maintains high accuracy for extreme wavenumber ratios. Figure~\ref{ExampleIm} plots the errors $\epsilon_\infty$ versus $N$ for several imaginary frequencies, demonstrating the accuracy of the proposed method for evanescent-wave problems.

\begin{figure}[htb]
	\centering
	\begin{tabular}{ccc}
		\includegraphics[scale=0.12]{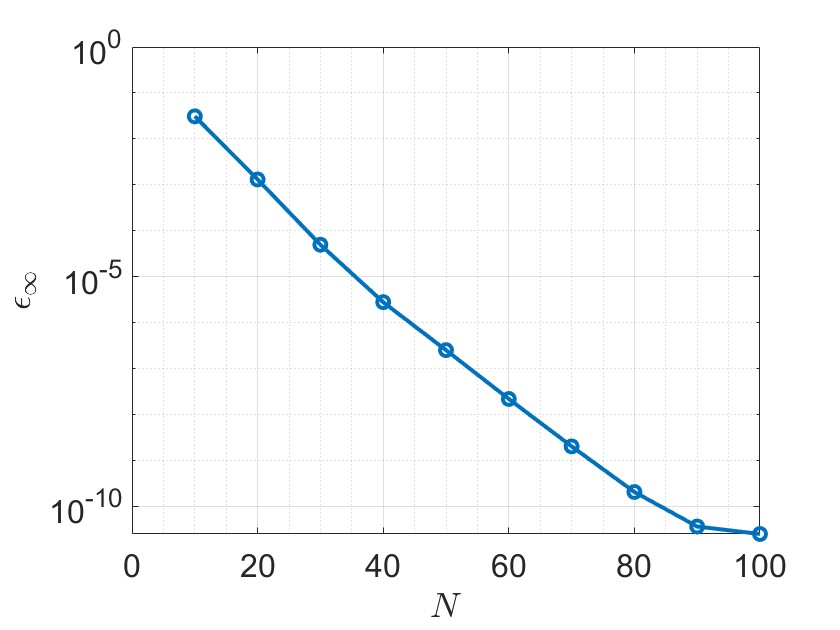} &
		\includegraphics[scale=0.12]{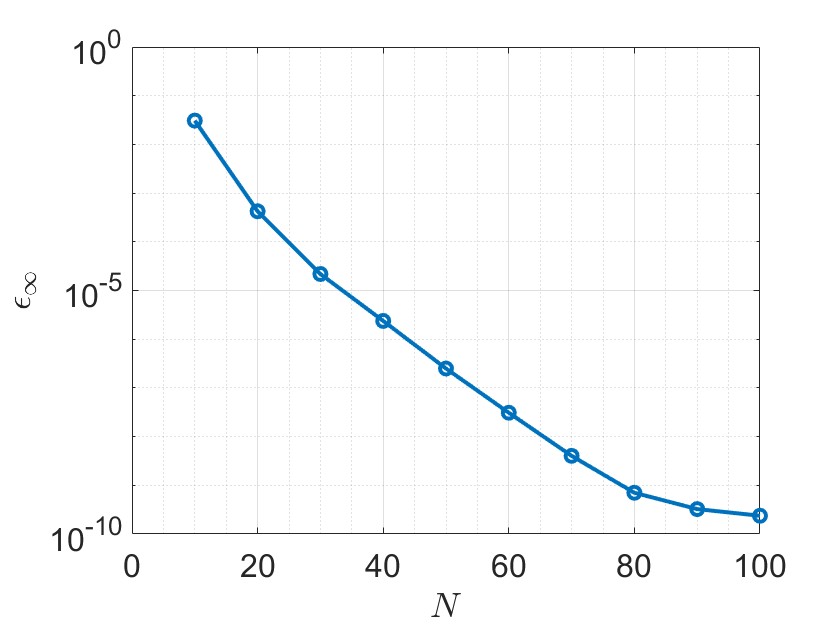} &
		\includegraphics[scale=0.12]{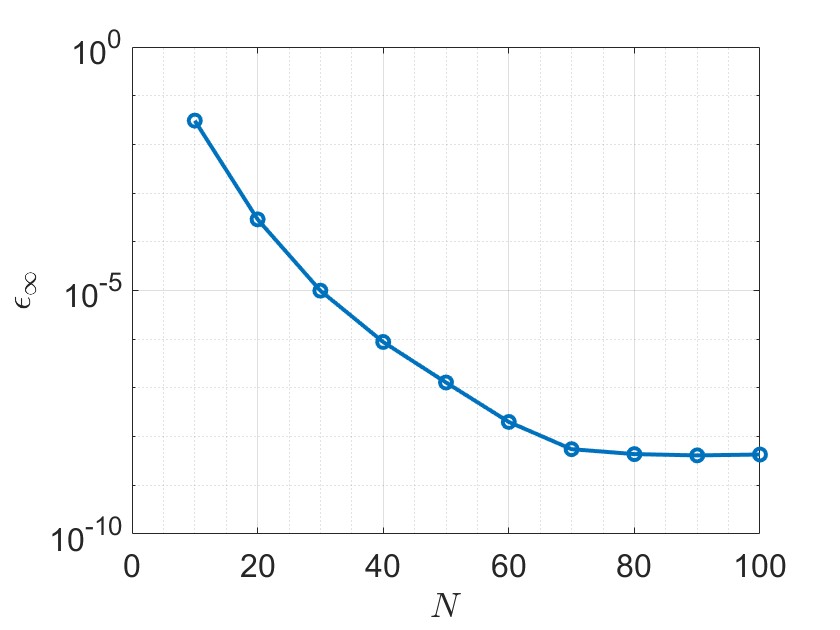} \\
		(a) $\frac{\lambda}{\mu}=10^2$  &(b) $\frac{\lambda}{\mu}=10^3$&(c) $\frac{\lambda}{\mu}=10^4$    
	\end{tabular}
	\caption{Numerical errors against $N$ for different Lamé parameter ratios with $d=0.1$.}
	\label{ExampleHB}
\end{figure}

\begin{figure}[htb]
	\centering
	\begin{tabular}{ccc}
		\includegraphics[scale=0.12]{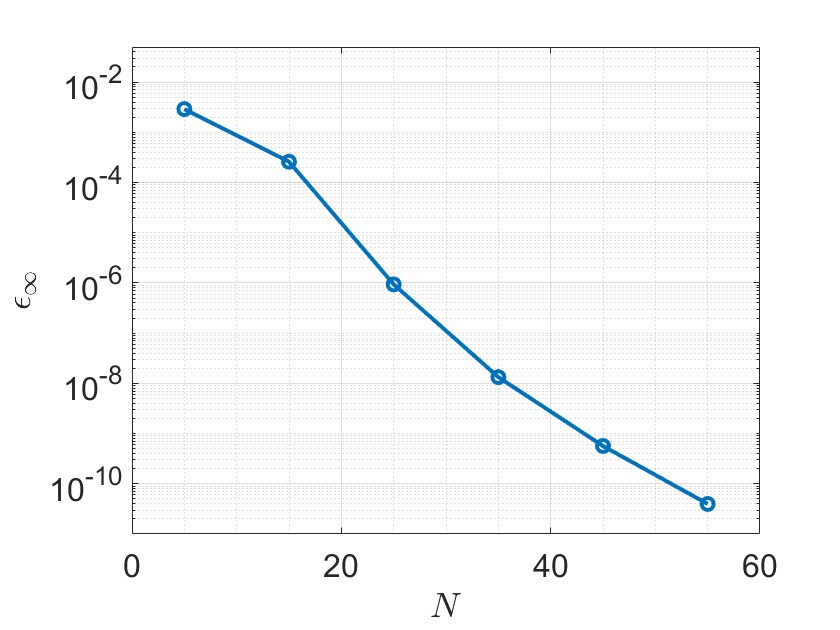} &
		\includegraphics[scale=0.12]{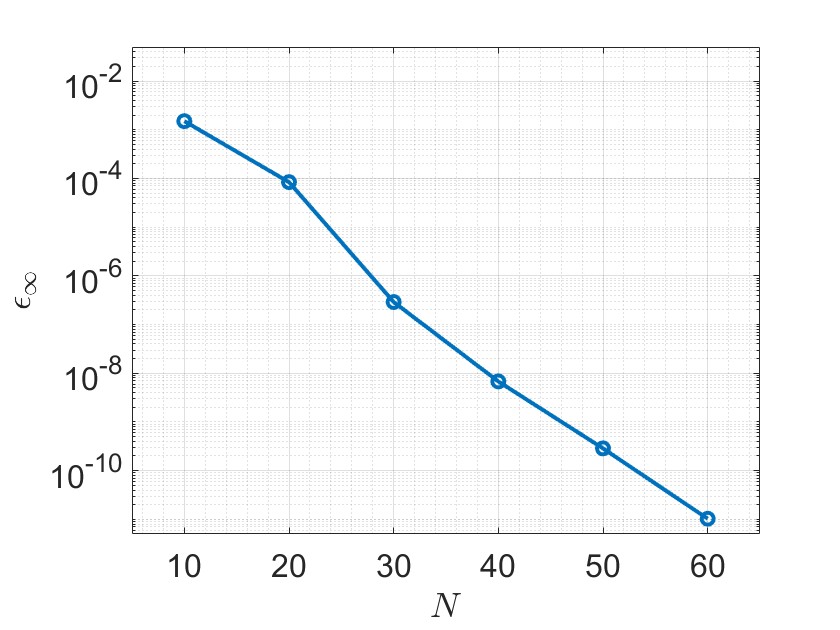} &
		\includegraphics[scale=0.12]{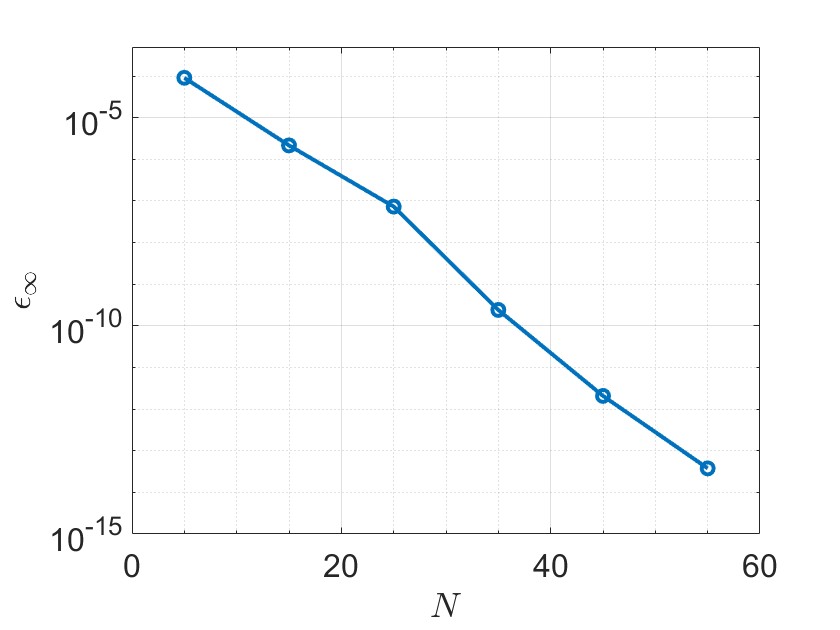} \\
		(a) $\omega=0.1{\rm i}$  &(b) $\omega=1{\rm i}$&(c) $\omega=10{\rm i}$    
	\end{tabular}
	\caption{Numerical errors against $N$ for different imaginary frequencies with $d=0.1$.}
	\label{ExampleIm}
\end{figure}

\begin{figure}[htb]
	\centering
	\begin{tabular}{ccc}
		\includegraphics[scale=0.12]{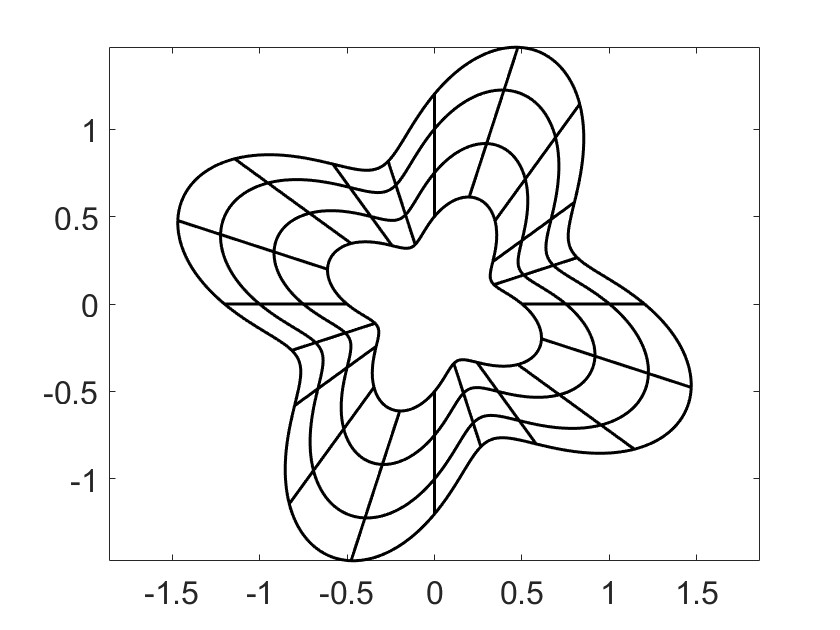} &
		\includegraphics[scale=0.12]{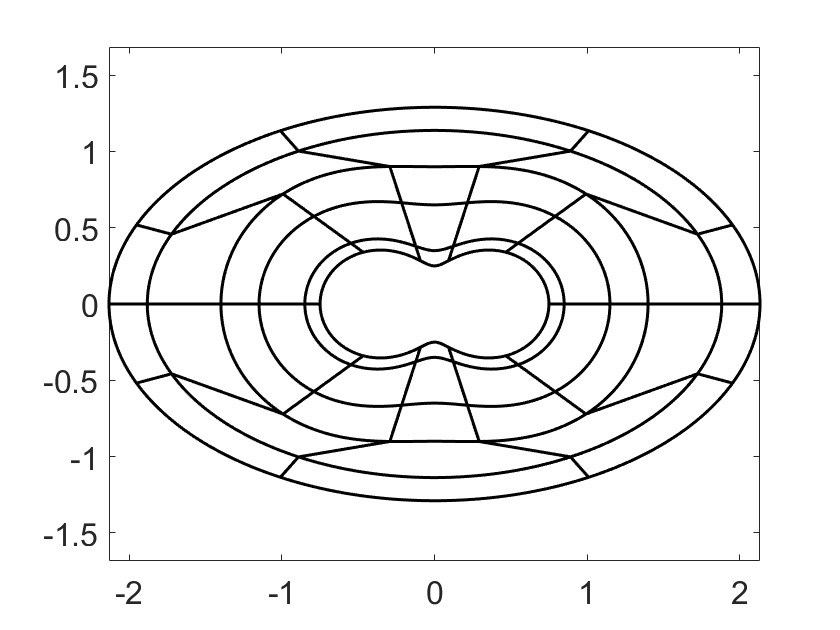} &
		\includegraphics[scale=0.12]{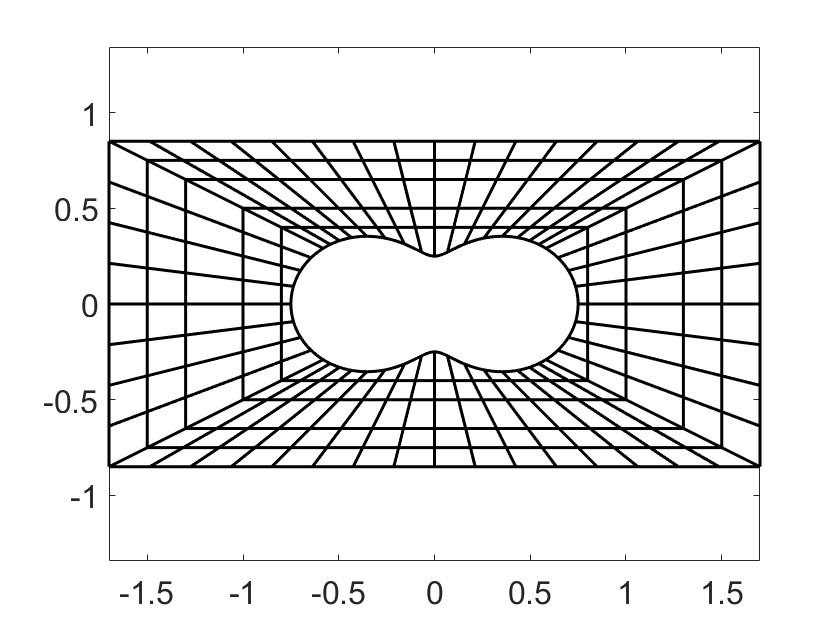} \\
		(a) Star-shaped layer   &(b) Elliptical layer &(c) Rectangular layer    
	\end{tabular}
	\caption{Obstacle shapes and mesh grids used in the numerical tests.}
	\label{Shapes}
\end{figure}

\subsection{Curvy four-pointed star-shaped layer}
\label{Example:5.2}
In the following examples, we demonstrate that the RCL method remains effective for a variety of layer geometries, including star-shaped, rectangular, and even nonconvex L-shaped layers. This illustrates the flexibility of the proposed approach.

We first examine a star-shaped scatterer $\Omega$, with its boundary $\Gamma$ defined by the following parametric representation
\begin{equation}
\label{R0}
R_0(\theta)=0.5+0.15 \sin (4(\theta+\pi / 4)), \quad \theta=[0,2 \pi) .
\end{equation}
We surround the scatterer with a star-shaped layer, and the parametric forms of the inner and outer boundaries of the layer are $R_1=2R_0(\theta)$ and $R_2=2.4R_0(\theta)$, respectively (see Figure \ref{Shapes} (a)). We take $\omega=30$, $l_1=20$ and $l_2=10$. Figure \ref{Examplestar} (a) denotes the exponential decay of relative errors in $\Omega_1$ as $N$ increases, which shows a good accuracy. In Figure \ref{Examplestar} (b), we plot $\Re(u_1^{\rm num})|_{\Omega_1}$ and $\Re(v_{\rm p}^{\rm num})|_{\Omega_2}$ with $N=50$. In Figure \ref{Examplestar} (c), we plot $\log(|\bm u^{\rm num}|)|_{\Omega_1}$ and $\log(|\phi_{\rm p}^{\rm num}|+|\phi_{\rm s}^{\rm num}|)|_{\Omega_2}$. It can be seen that the numerical solutions decay rapidly in the RCL region $\Omega_2$. 
\begin{figure}[htb]
	\centering
	\begin{tabular}{ccc}
		\includegraphics[scale=0.12]{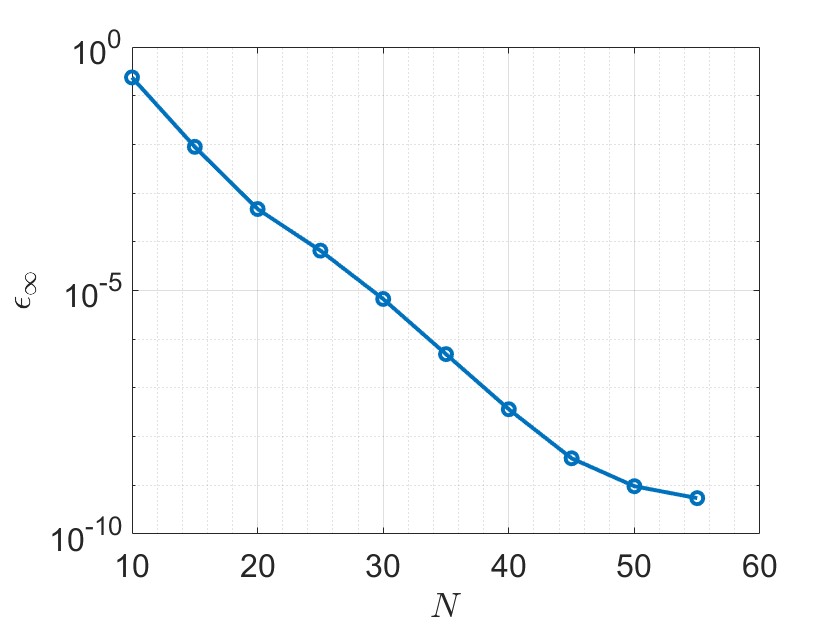} &
		\includegraphics[scale=0.12]{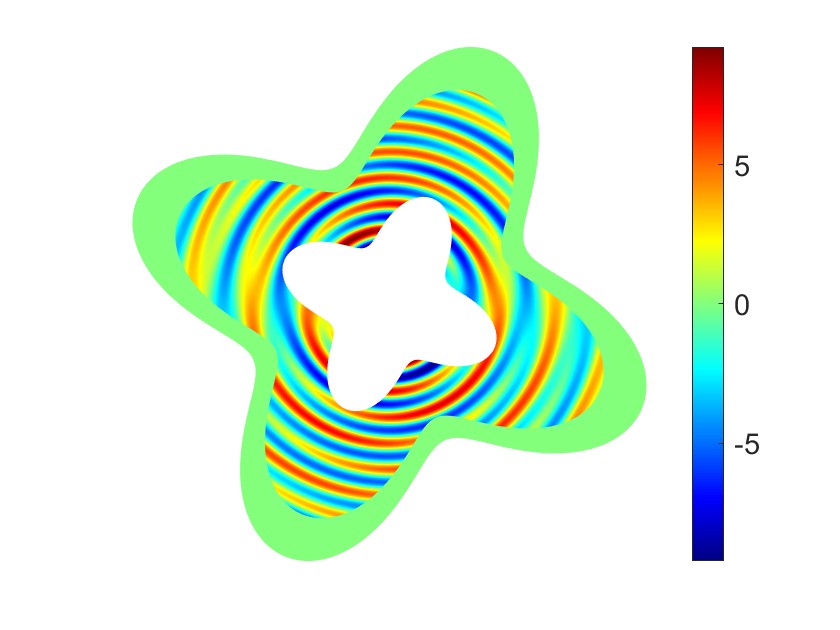} &
		\includegraphics[scale=0.12]{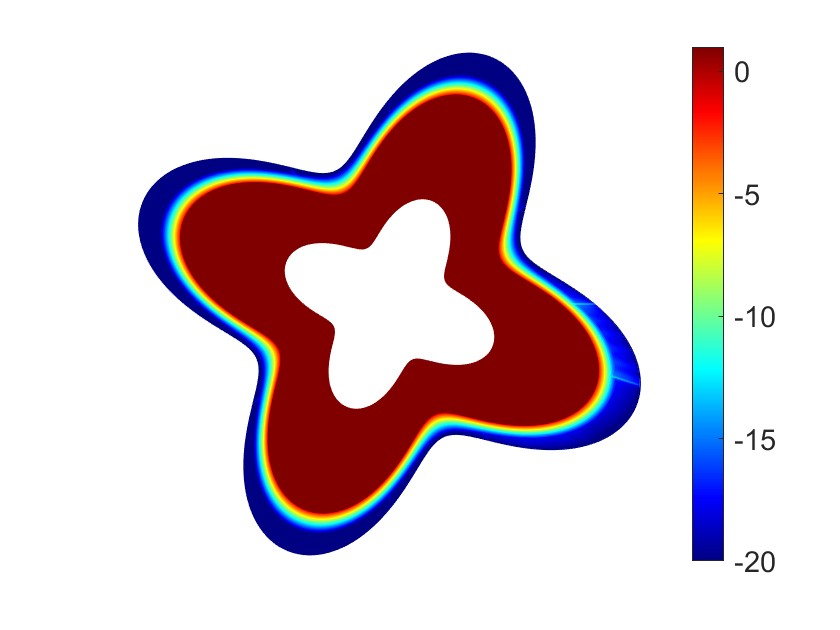} \\
		(a) errors against $N$&(b) $\Re(u_1^{\rm num})|_{\Omega_1}$  & (c) $\log(|\bm u^{\rm num}|)$  
	\end{tabular}
	\caption{Curvy Four-pointed Star-shaped scatterer and RCL layer.}
	\label{Examplestar}
\end{figure}

\subsection{Elliptical layer}
\label{Example:5.3}

Consider a peanut-shaped scatterer $\Omega$ with the parametric form of the boundary:
\begin{equation}
\label{EpR0}
R_0(\theta)=0.5+0.25 \sin (2(\theta+\pi / 4)), \quad \theta=[0,2 \pi).
\end{equation}
As shown in Figure \ref{Shapes} (b), we truncate the unbounded domain by an elliptical artificial layer. We take $\omega=10$, $l_1=40$ and $l_2=10$. It can be seen from Figure \ref{ExampleEp} (a) that the relative errors decay exponentially with the increase of $N$. The real parts of the first component of the exact and numerical fields are illustrated in Figures \ref{ExampleEp} (b)-(c). It is observed that the numerical results coincide perfectly with the exact solutions in $\Omega_1$, while rapidly vanishing in $\Omega_2$.
\begin{figure}[htb]
	\centering
	\begin{tabular}{ccc}
		\includegraphics[scale=0.12]{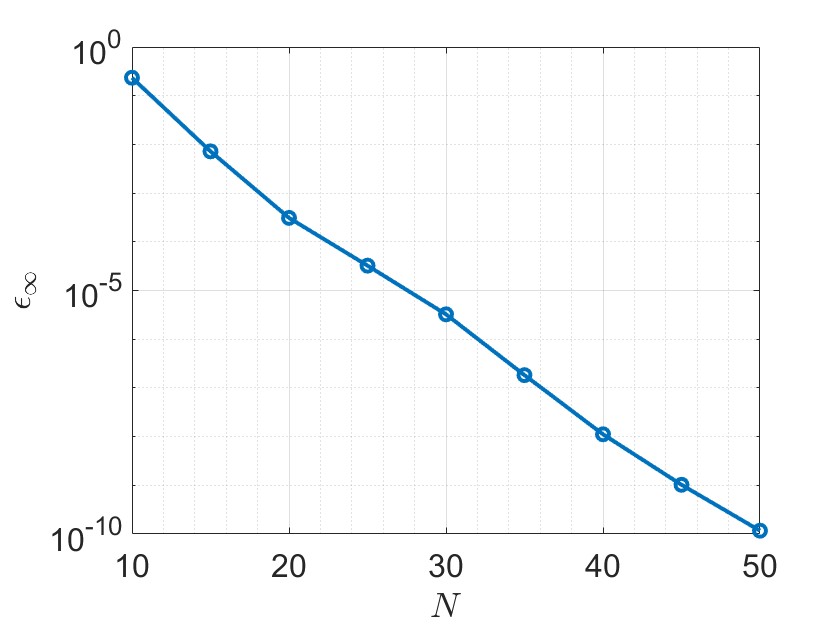} &
		\includegraphics[scale=0.12]{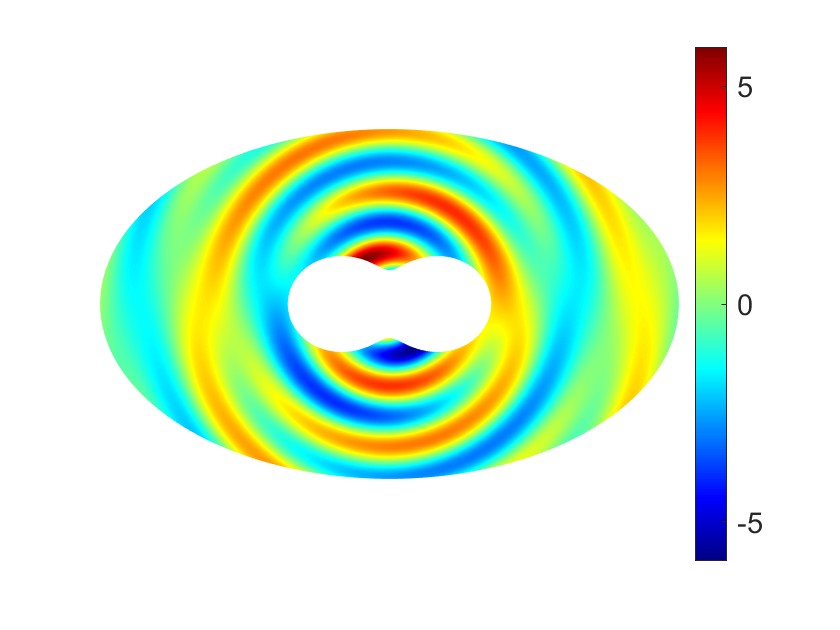} &
		\includegraphics[scale=0.12]{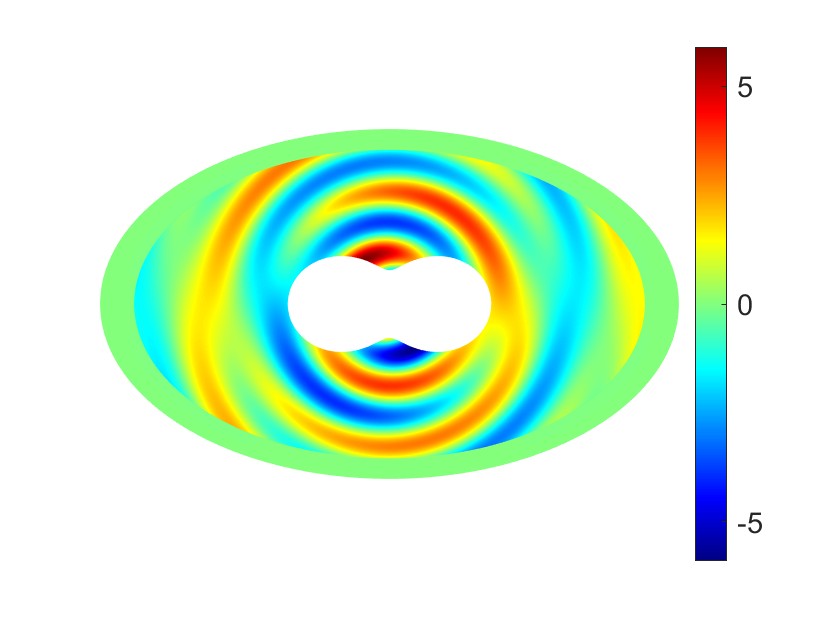}\\
		(a) errors against $N$&(b) $\Re(u^{\rm ex}_1)$ & (c) $\Re(u^{\rm num}_1)$ 
	\end{tabular}
	\caption{Peanut-shaped scatterer with elliptical RCL layer. The simulation results are obtained with $R_2(\theta)=\frac{17}{15}R_1(\theta)$.}
	\label{ExampleEp}
\end{figure}

\subsection{Rectangular layer}
\label{Example:5.4}

In this example, we consider the scattering of an incident plane wave
\begin{equation*}
\bm{u}^{\rm inc}=\bm d e^{{\rm i} k_s\bm x \cdot \bm d}, \quad \bm d=(1,0)^{\top},
\end{equation*}
which yields that the boundary condition
\begin{equation*}
\bm g=\bm d e^{{\rm i} k_s\bm x \cdot \bm d}, \quad \bm d=(1,0)^{\top}.
\end{equation*}
Then we surround the same scatterer \eqref{EpR0} with a rectangular layer, and the boundary $\Gamma_2$ is a rectangle with four vertices: $(-1.5,-0.75)$, $(-1.5,0.75)$, $(1.5,-0.75)$ and $(1.5,0.75)$ (see Figure \ref{Shapes} (c)). Here, we choose $\omega=30$, $l_1=192$ and $l_2=48$. Figure \ref{Examplehom} (a) presents the relative errors $\epsilon_{\infty}$ with respect to different $N$, which clearly demonstrates the efficiency the proposed method. We depict the real parts of the two component of of numerical solutions in Figures \ref{Examplehom} (b)-(c).

As a comparison, we consider the exterior scattering problem with a locally inhomogeneous medium. All numerical settings remain unchanged, except that the function $\theta(\bm x)$ in $\Omega_1$ is replaced by a shifted Gaussian function 
\begin{equation*}
n(\bm{x})=1+c_0 \exp (-\tfrac{(x-x_0)^2+(y-y_0)^2}{2 c_1^2}) 
\end{equation*}
with $x_0=0, y_0=0.3, c_0=2$, and $c_1=0.06$. Figure \ref{Exampleinhom} (a) shows that the maximum relative error is depicted in , which demonstrate that the proposed RCL technique is accurate and robust for various scatterers with locally inhomogeneous media. The real parts of the two component of of numerical solutions is depict in Figure \ref{Exampleinhom} (b)-(c). Compared with homogeneous medium, it is observable that wave oscillations increase over the upper-middle portion of the peanut scatterer as a result of the surrounding inhomogeneity.

\begin{figure}[htb]
	\centering
	\begin{tabular}{ccc}
		\includegraphics[scale=0.12]{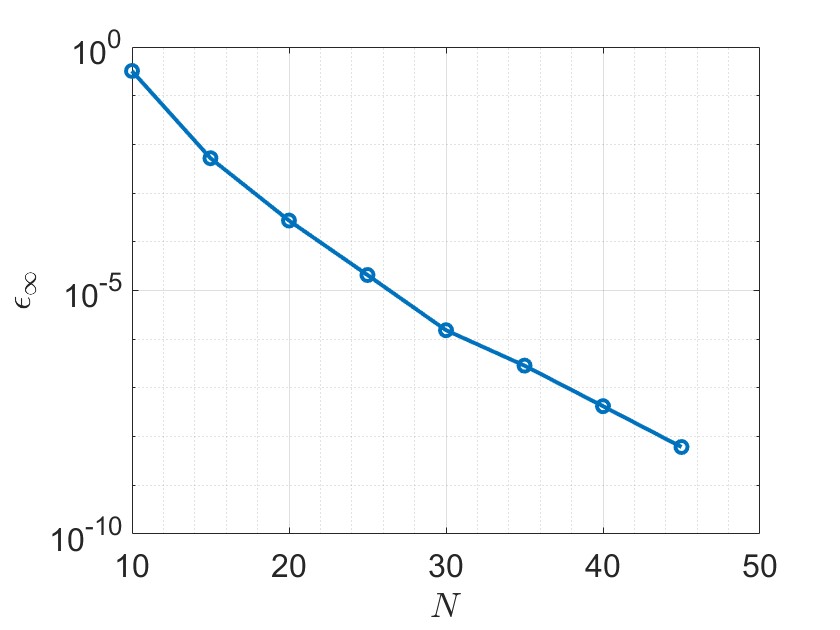} &
		\includegraphics[scale=0.12]{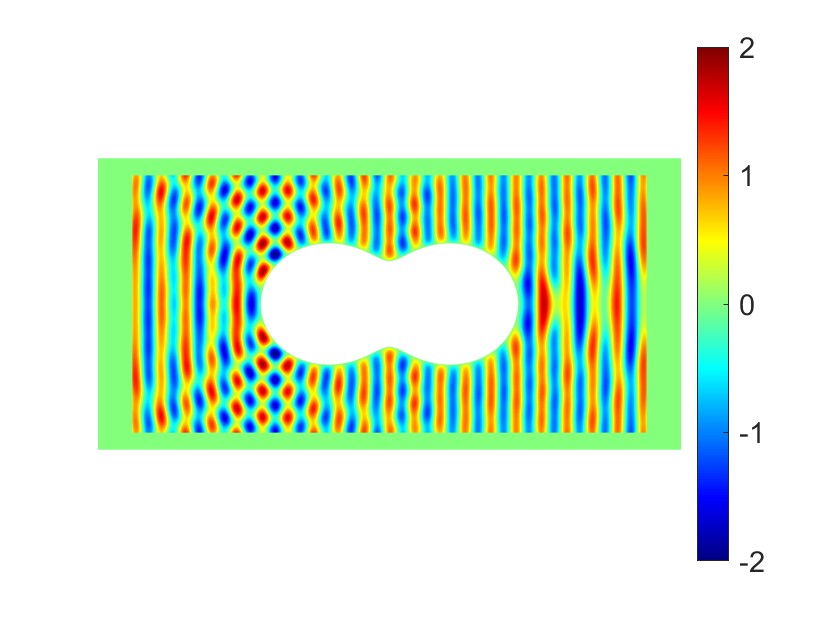} &
		\includegraphics[scale=0.12]{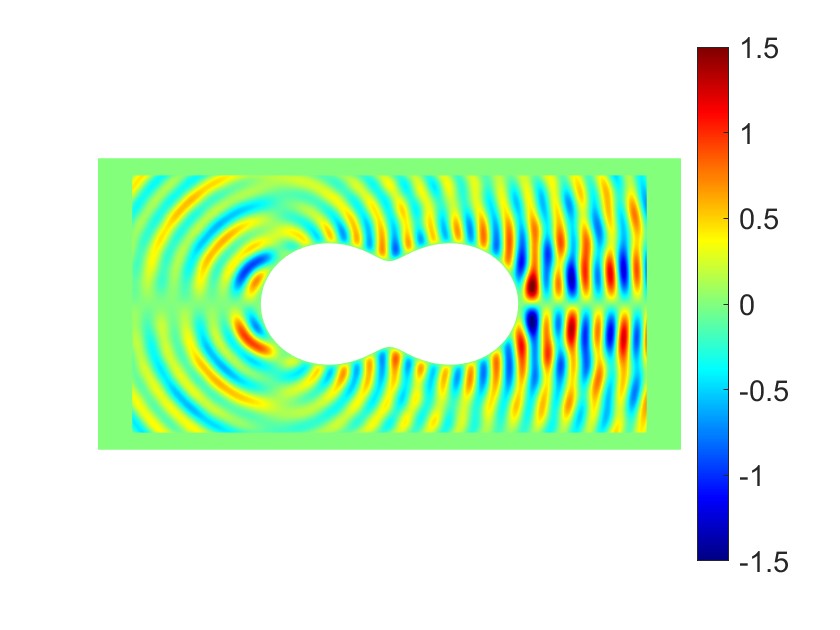} \\
		(a) errors against $N$  &(b) $\Re(u_1^{\rm num})$  &(c) $\Re(u_2^{\rm num})$    \\
	\end{tabular}
	\caption{Peanut-shaped scatterer with a rectangluar PAL layer. The simulation results are obtained with $R_2(\theta)=\frac{17}{15}R_1(\theta)$.}
	\label{Examplehom}
\end{figure}

\begin{figure}[htb]
	\centering
	\begin{tabular}{ccc}
		\includegraphics[scale=0.12]{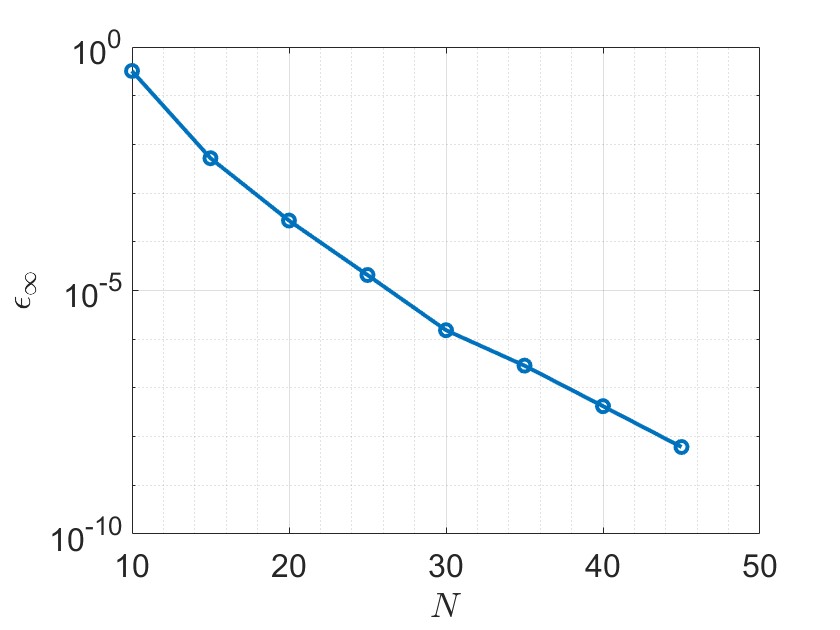} &
		\includegraphics[scale=0.12]{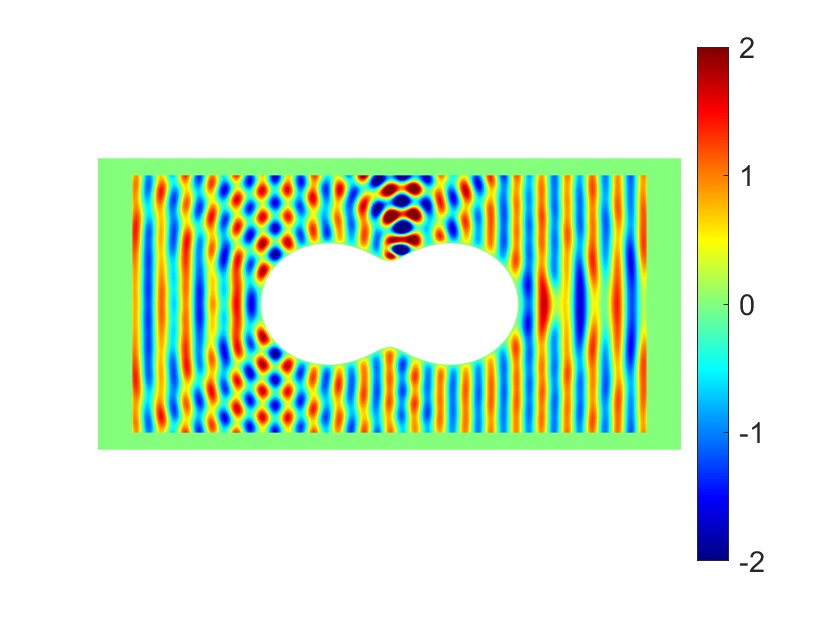} &
		\includegraphics[scale=0.12]{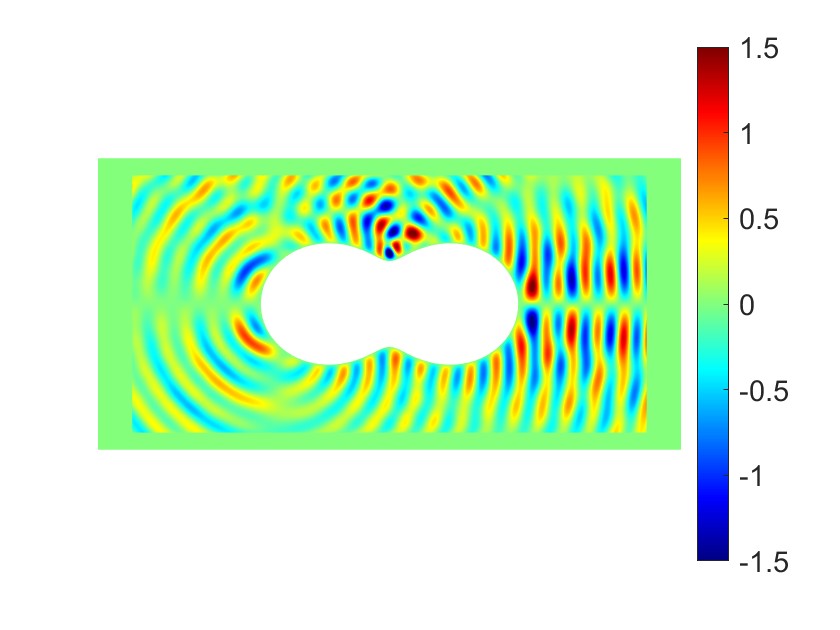} \\
		(a) errors against $N$  &(b) $\Re(u_1^{\rm num})$  &(c) $\Re(u_2^{\rm num})$    \\
	\end{tabular}
	\caption{Scattering problem with a locally inhomogeneous medium.}
	\label{Exampleinhom}
\end{figure}

\subsection{L-shaped scatterer}
\label{Example:5.5}

Finally, we consider the elastic problem of scattering by an L-shaped obstacle. Let $\omega=10$, $l_1=144$, $l_2=48$ and $N=30$. As shown in Figure \ref{ExampleLshape} (a), we only need to set up the RCL layer in a small neighborhood of the physical domain, which can effectively save computational cost. In Figure \ref{ExampleLshape} (b), we plot $\log(|\bm u^{\rm num}|)|_{\Omega_1}$ and $\log(|\phi_{\rm p}^{\rm num}|+|\phi_{\rm s}^{\rm num}|)|_{\Omega_2}$, which shows that the numerical solutions decay rapidly in the RCL region $\Omega_2$. Figures \ref{ExampleLshape} (c)-(f) compare the real parts of the exact and numerical solutions for two components of $\bm u$. It is observed that the numerical solutions match perfectly with the exact solutions in the region $\Omega_1$, whereas in $\Omega_2$, the real parts of $v_{\rm p}$ and $v_{\rm s}$ decay rapidly. This demonstrates that, for scattering problems involving non-convex domains, the proposed method can still performs well.

\begin{figure}[htb]
	\centering
	\begin{tabular}{ccc}
		\includegraphics[scale=0.12]{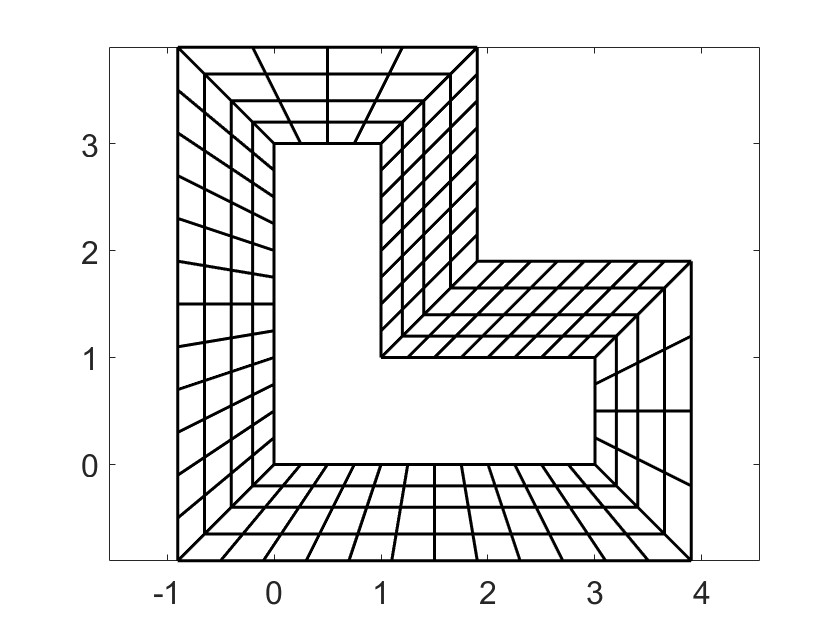} &
		\includegraphics[scale=0.12]{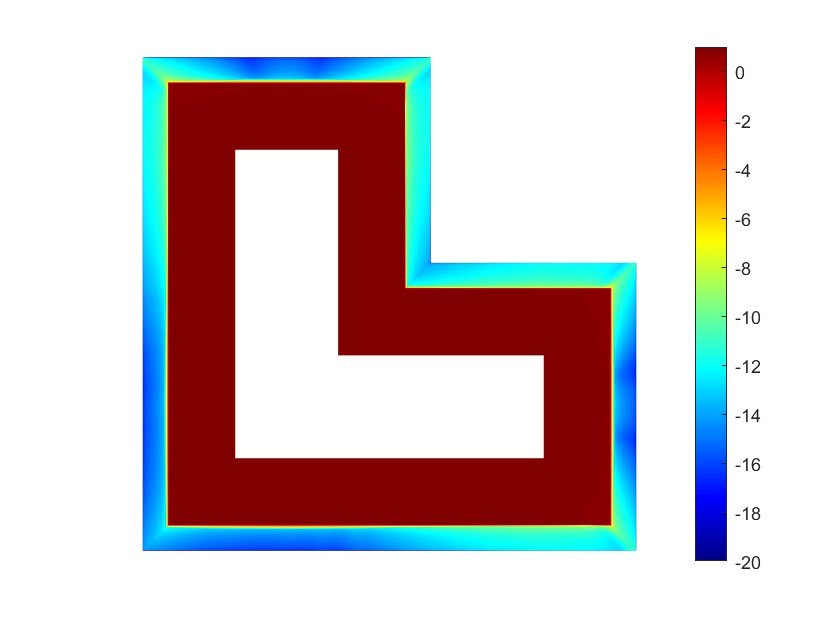} &
		\includegraphics[scale=0.12]{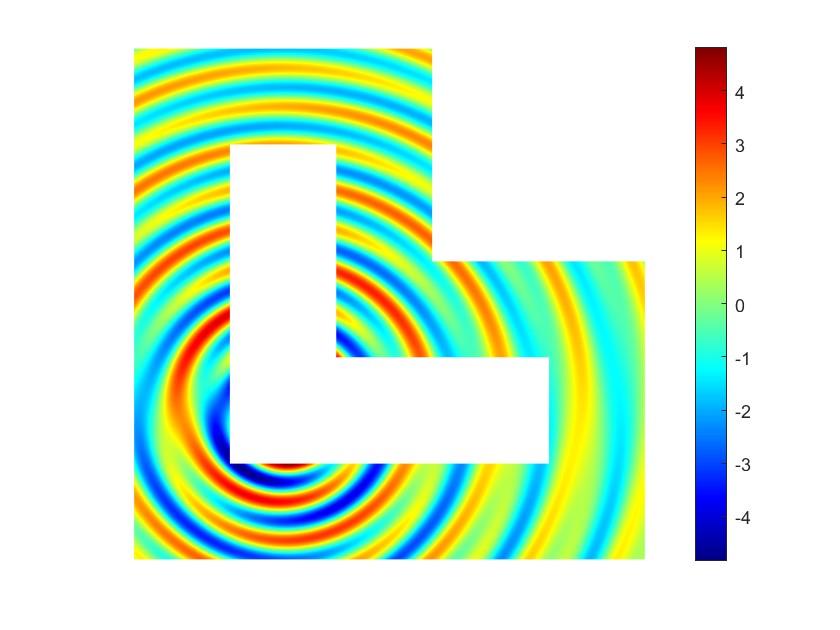} \\
		(a) illustration of mesh grids  &(b) $\log(|\bm u^{\rm num}|)$&(c) $\Re(u_1^{\rm exa})$    \\
		\includegraphics[scale=0.12]{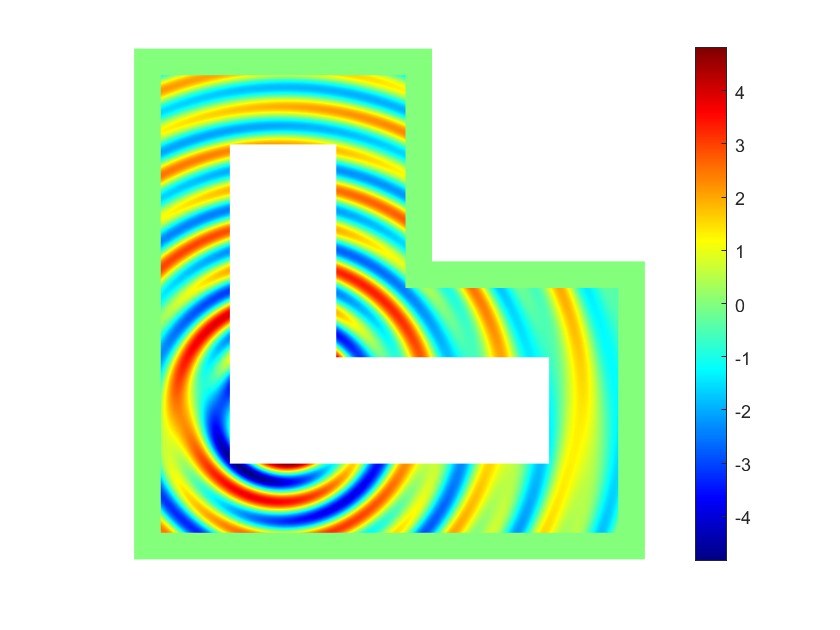} &
		\includegraphics[scale=0.12]{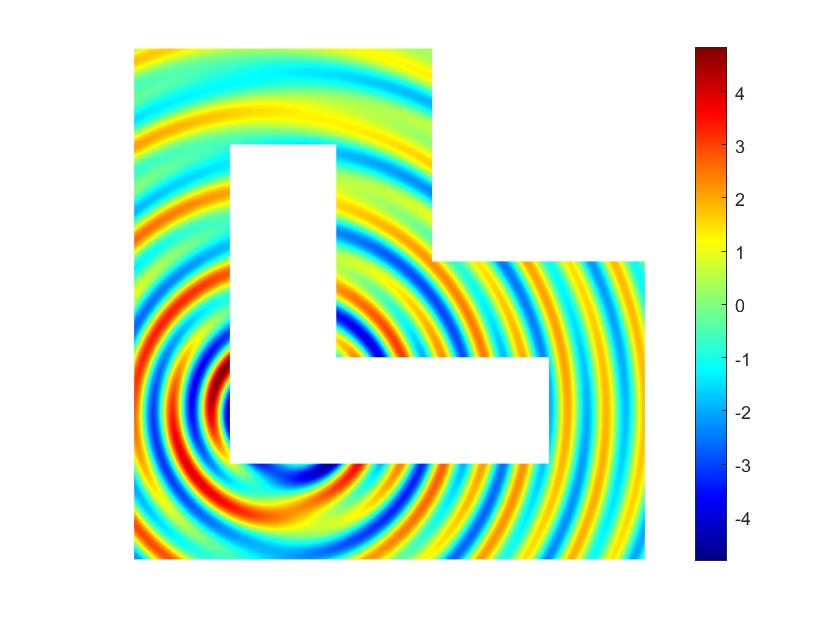} &
		\includegraphics[scale=0.12]{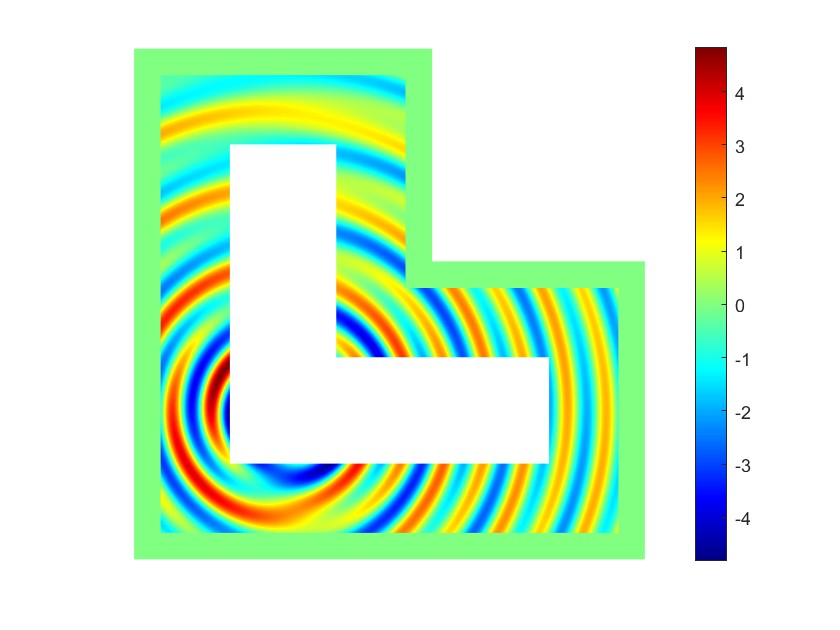} \\
		(d) $\Re(u_1^{\rm num})$   &(e) $\Re(u_2^{\rm exa})$ &(f) $\Re(u_2^{\rm num})$   \\
	\end{tabular}
	\caption{Comparison between the exact and numerical solutions for waves scattered by an L-shaped scatterer with $\omega=10$ and $N=30$.}
	\label{ExampleLshape}
\end{figure}

\section{Conclusion}

In this paper, we developed a real compressed layer method for domain reduction of time-harmonic elastic wave scattering problems in inhomogeneous media. Since the direct use of real compression is not suitable for elastic waves, we apply the Helmholtz decomposition only in the exterior homogeneous layer to decompose the elastic displacement field into compressional and shear wave components. After real compression, the resulting oscillatory modes can be explicitly extracted through suitable phase-extraction substitutions. We proved the well-posedness and exponential convergence of the resulting RCL problem, and developed a high-order spectral element discretization. Numerical experiments show that the proposed method is accurate, robust, and nonreflecting for high- and low-frequency waves, large wavenumber ratios, imaginary wavenumbers, and different layer geometries.

Beyond the present elastic scattering setting, the proposed RCL framework provides a promising strategy for more complicated multi-wave and coupled systems. Its capability to handle low-frequency waves and complex wavenumber components is particularly relevant to time-domain solvers based on convolution quadrature~\cite{BLM11}, where a sequence of complex-frequency problems must be solved, as well as to frequency-time hybrid solvers requiring accurate low-frequency solutions~\cite{ABL20}. Similar ideas may also be useful for biharmonic, thermoelastic, poroelastic, and other coupled wave models, where multiple wave modes and complex or imaginary wavenumbers naturally arise. Extensions to three-dimensional problems, layered high-contrast media, and anisotropic media are also of interest.

\appendix
\section*{Appendix. Proof of Lemma 2.3}
\label{proofoflemma}

We assume that 
\begin{align}
\label{uR}
&\bm u=\sum_{n=-\infty}^{\infty}(u_n^r(r,\theta)\bm e_r+u_n^\theta(r,\theta)\bm e_\theta)e^{{\rm i}n\theta},\\[0.4em]
\label{PS}
&\phi_{\rm t}(r, \theta)=\sum_{n=-\infty}^{\infty} \phi_{{\rm t},n}(r) \rm{e}^{\rm{i} n \theta},\quad t=p,s.
\end{align}
Using \eqref{HD}, we can obtain that
\begin{equation}
\label{upc}
u_n^r=\tfrac{\partial \phi_{{\rm p},n}}{\partial r}+\tfrac{1}{r} \tfrac{\partial \phi_{{\rm s},n}}{\partial \theta}, \quad u_n^\theta=\tfrac{1}{r} \tfrac{\partial \phi_{{\rm p},n}}{\partial \theta}-\tfrac{\partial \phi_{{\rm s},n}}{\partial r}.
\end{equation}
In polar coordinates, the Laplace operator has the expression
\begin{equation}
\label{PL}
\Delta=\tfrac{1}{r} \tfrac{\partial}{\partial r}(r \tfrac{\partial}{\partial r})+\tfrac{1}{r^2} \tfrac{\partial^2}{\partial \theta^2}.
\end{equation}
Substituting \eqref{PS} into \eqref{PL} yields 
\begin{equation}
\label{Bssel}
[\tfrac{1}{r} \tfrac{d}{d r}(r \tfrac{d}{d r})+(k_{\rm t}^2-\tfrac{n^2}{r^2})] \phi_{{\rm t},n}(r)=0,\quad t=p,s.
\end{equation}
According to~\cite{OLBC10}, the outgoing solution of \eqref{Bssel} takes the following form
\begin{equation}
\label{vpR}
\phi_{{\rm t},n}=b_{{\rm t},n}H_n^{(1)}(k_{\rm t} r).
\end{equation}
According to \eqref{upc}, we have
\begin{equation}
\label{urtheta}
\begin{aligned}
& u_n^r(r, \theta)=k_{\rm p} b_{{\rm p},n} H_n^{(1)^{\prime}}(k_{\rm p} r)+\tfrac{\rm{i} n}{r}b_{{\rm s},n} H_n^{(1)}(k_{\rm s} r) , \\
& u_n^\theta(r, \theta)=\tfrac{\rm{i} n}{r}b_{{\rm p},n} H_n^{(1)}(k_{\rm p} r)-k_{\rm s} b_{{\rm s},n} H_n^{(1)^{\prime}}(k_{\rm s} r).
\end{aligned}
\end{equation}
In the polar coordinates, the boundary data $\bm g$ has the Fourier series expansion
\begin{equation}
\label{gR}
\bm g=\sum_{n=-\infty}^{\infty}(g_n^r(r,\theta)\bm e_r+g_n^\theta(r,\theta)\bm e_\theta)e^{{\rm i}n\theta}.
\end{equation}
Then we can obtain the system for the Fourier coefficients 
\begin{equation*}
\begin{bmatrix}
k_{\rm p}\, H_{n}^{(1)\,'}\!\big(k_{\rm p} a\big) & \tfrac{{\rm i} n}{a}\, H_{n}^{(1)}\!\big(k_{\rm s} a\big)\\[6pt]
\tfrac{{\rm i} n}{a}\, H_{n}^{(1)}\!\big(k_{\rm p} a\big) & -\,k_{\rm s}\, H_{n}^{(1)\,'}\!\big(k_{\rm s} a\big)
\end{bmatrix}
\begin{bmatrix} b_{{\rm p},n} \\[6pt] b_{{\rm s},n} \end{bmatrix}
=
\begin{bmatrix}  g_{r,n} \\[6pt]  g_{s,n} \end{bmatrix},
\end{equation*}
which leads to 
\begin{equation}
\label{anbn}
\begin{split}
& b_{{\rm p},n} = -\tfrac{\,k_{\rm s}\, H_{n}^{(1)'}(k_{\rm s} a)  g_{r,n} + \tfrac{{\rm i} n}{a} H_{n}^{(1)}(k_{\rm s} a)  g_{\theta,n}}{D_n},\\
& b_{{\rm s},n} = \tfrac{k_{\rm p} H_{n}^{(1)'}(k_{\rm p} a)  g_{\theta,n}- \tfrac{{\rm i} n}{a} H_{n}^{(1)}(k_{\rm p} a)   g_{r,n}}{D_n},
\end{split}
\end{equation}
where 
\begin{equation}
\label{Dmdefn}
D_n := -k_{\rm p} k_{\rm s} H_{n}^{(1)'}(k_{\rm p} a) H_{n}^{(1)'}(k_{\rm s} a)+ \tfrac{n^2}{a^2} H_{n}^{(1)}(k_{\rm p} a) H_{n}^{(1)}(k_{\rm s} a).
\end{equation}
In particular, for $n=0,$ we have
\[
b_{{\rm p},0}=\tfrac{ g_{r,0}} {k_{\rm p}H_{0}^{(1)'}(k_{\rm p} a)} , 
\quad b_{{\rm s},0}=-\tfrac{ g_{\theta,0}} {k_{\rm s} H_{0}^{(1)'}(k_{\rm s} a)}.
\]
We can rewrite \eqref{Dmdefn} as
\begin{equation*}
D_n=H_n^{(1)}(k_{\rm p} a) H_n^{(1)}(k_{\rm s} a)\Big[\tfrac{n^2}{a^2}-k_{\rm p} k_{\rm s} \tfrac{H_n^{(1) \prime}(k_{\rm p} a)}{H_n^{(1)}(k_{\rm p} a)} \tfrac{H_n^{(1) \prime}(k_{\rm s} a)}{H_n^{(1)}(k_{\rm s} a)}\Big].
\end{equation*}
It can be seen from \cite{SW2007} that
\begin{equation*}
\Re(\tfrac{H_n^{(1) \prime}(k_{\rm t} a)}{H_n^{(1)}(k_{\rm t} a)})<0,\quad \Im(\tfrac{H_n^{(1) \prime}(k_{\rm t} a)}{H_n^{(1)}(k_{\rm t} a)})>0,
\end{equation*}
which implies that
\begin{equation*}
|D_n|>k_{\rm p}k_{\rm s}|H_n^{(1)}(k_{\rm p} a) H_n^{(1)}(k_{\rm s} a)||\Im (\tfrac{H_n^{(1) \prime}(k_{\rm p} a)}{H_n^{(1)}(k_{\rm p} a)} \tfrac{H_n^{(1) \prime}(k_{\rm s} a)}{H_n^{(1)}(k_{\rm s}a)})|>0,\quad n\ge 1.
\end{equation*}
It is obvious that $D_0\neq0$. Therefore, the solutions \eqref{anbn} are well-defined. In view of the identity $H_{-n}^{(1)}(z)=(-1)^n H_n^{(1)}(z)$, we have
\begin{equation}
\label{phit}
\phi_{\rm t}(r, \theta)=\sum_{n=1}^{\infty} c_{{\rm t},n}(\theta)H_n^{(1)}(k_{\rm t} r) ,\quad t=p,s,
\end{equation}
where
\begin{equation*}
c_{{\rm t},0}=b_{{\rm t},0},\quad c_{{\rm t},n}(\theta)=b_{{\rm t},n}\rm{e}^{\rm{i} n \theta}+(-1)^nb_{{\rm t},-n}\rm{e}^{-\rm{i} n \theta},\quad \text{for}\quad n>0.
\end{equation*}
On the other hand, we know from~\cite{OLBC10} that
\begin{equation}
\label{HN}
H_n^{(1)}(k_{\rm t} r)=H_0^{(1)}(k_{\rm t} r) P_n(\tfrac{1}{k_{\rm t} r})+H_1^{(1)}(k_{\rm t} r) Q_n(\tfrac{1}{k_{\rm t} r}),
\end{equation}
where $P_n$, $Q_n$ are the Lommel polynomials~\cite{W95} given by
\begin{equation*}
\begin{aligned}
&P_n(x)=-\sum_{j=0}^{\lfloor(n-2) / 2\rfloor}(-1)^j \frac{(n-2-j)!(n-j-1)!}{j!(j+1)!(n-2-2 j)!}(2 x)^{n-2-2 j},\\
&Q_n(x)=\sum_{j=0}^{\lfloor(n-1) / 2\rfloor}(-1)^j \frac{((n-1-j)!)^2}{(j!)^2(n-1-2 j)!}(2 x)^{n-1-2 j}.
\end{aligned}
\end{equation*}
Rearranging $P_n$ and $Q_n$ in powers of $\tfrac{1}{r}$ yields
\begin{equation*}
P_n(\tfrac{1}{k_{\rm t} r})  =\sum_{m=0}^{n} p_{n, m} (\tfrac{1}{k_{\rm t}r})^m, \quad Q_n(\tfrac{1}{k_{\rm t} r})  =\sum_{m=0}^{n} q_{n, m} (\tfrac{1}{k_{\rm t}r})^m .
\end{equation*}
A combination of \eqref{phit} and \eqref{HN} leads to
\begin{equation*}
\phi_{\rm t}(r, \theta)=H_0^{(1)}(k_{\rm t} r) \sum_{l=0}^{\infty} \frac{h_{{\rm t},l}(\theta)}{(k_{\rm t} r)^l}+H_1^{(1)}(k_{\rm t} r) \sum_{l=0}^{\infty} \frac{q_{{\rm t},l}(\theta)}{(k_{\rm t} r)^l},
\end{equation*}
where
\begin{equation*}
h_{{\rm t},l}= \sum_{n=0}^{\infty}c_{{\rm t},n}(\theta) p_{n, l},\quad q_{{\rm t},l}= \sum_{n=0}^{\infty}c_{{\rm t},n}(\theta) q_{n, l}.
\end{equation*}
In fact, given $h_{{\rm t},0}$ and $q_{{\rm t},0}$, the higher-order coefficients $h_{{\rm t},n}$ and $q_{{\rm t},n}$ can be determined recursively by requiring the net coefficients of $H_0^{(1)}$ and $H_1^{(1)}$ to vanish identically.

Using the formula 
\begin{equation*}
\tfrac{d H_n^{(1)}(z)}{d z}=\tfrac{n}{z} H_n^{(1)}(z)- H_{n+1}^{(1)}(z),
\end{equation*}
we have
\begin{equation}
\label{uR1}
\bm u=\sum_{n=0}^{\infty}[\bm c_n^{\rm p}(\theta)H_n^{(1)}(k_{\rm p}r)+\bm c_n^{\rm s}(\theta)H_n^{(1)}(k_{\rm s}r)],
\end{equation}
where
\begin{equation*}
\begin{aligned}
&\bm c_n^{\rm p}(\theta)=[\tfrac{nb_{{\rm p},n}}{r}\rm{e}^{\rm{i} n \theta}-k_{\rm p}b_{{\rm p},n-1}\rm{e}^{\rm{i}( n-1) \theta}+(-1)^{n+1}\tfrac{nb_{{\rm p},-n}}{r}\rm{e}^{-\rm{i} n \theta}\\
&+(-1)^{n+1}k_{\rm p}b_{{\rm p},1-n}\rm{e}^{\rm{i} (1-n) \theta}]\bm e_r+[\tfrac{{\rm i}nb_{{\rm p},n}}{r}\rm{e}^{\rm{i} n \theta}+(-1)^{n+1}\tfrac{{\rm i}nb_{{\rm p},-n}}{r}\rm{e}^{-\rm{i} n \theta}]\bm e_\theta,\\
&\bm c_n^{\rm s}(\theta)=[\tfrac{{\rm i}nb_{{\rm s},n}}{r}\rm{e}^{\rm{i} n \theta}+(-1)^{n+1}\tfrac{{\rm i}nb_{{\rm s},-n}}{r}\rm{e}^{-\rm{i} n \theta}]\bm e_r-[\tfrac{nb_{{\rm s},n}}{r}\rm{e}^{\rm{i} n \theta}\\
&-k_{\rm s}b_{{\rm s},n-1}\rm{e}^{\rm{i} (n-1) \theta}-(-1)^{n}\tfrac{nb_{{\rm s},-n}}{r}\rm{e}^{-\rm{i} n \theta}-(-1)^{n}k_{\rm s}b_{{\rm s},1-n}\rm{e}^{\rm{i} (1-n) \theta}]\bm e_\theta.
\end{aligned}
\end{equation*}
Combining \eqref{uR1} and \eqref{HN} gives
\begin{equation*}
\begin{aligned}
\bm u(r, \theta)&=H_0^{(1)}(k_{\rm p} r) \sum_{l=0}^{\infty} \tfrac{\bm F_{{\rm p},l}(\theta)}{(k_{\rm p} r)^l}+H_1^{(1)}(k_{\rm p} r) \sum_{l=0}^{\infty} \tfrac{\bm G_{{\rm p},l}(\theta)}{(k_{\rm p} r)^l}\\
&+H_0^{(1)}(k_{\rm s} r) \sum_{l=0}^{\infty} \tfrac{\bm F_{{\rm s},l}(\theta)}{(k_{\rm s} r)^l}+H_1^{(1)}(k_{\rm s} r) \sum_{l=0}^{\infty} \tfrac{\bm G_{{\rm s},l}(\theta)}{(k_{\rm s} r)^l},
\end{aligned}
\end{equation*}
where
\begin{equation*}
\bm F_{{\rm t},l}= \sum_{n=0}^{\infty}c_n^{{\rm t}}(\theta) p_{n, l},\quad \bm G_{{\rm t},l}= \sum_{n=0}^{\infty}c_n^{{\rm t}}(\theta) q_{n, l},\quad \rm t=p,s.
\end{equation*}
For further details on the convergence and differentiability of these series, we refer the reader to \cite{K61,CL06}. This completes the proof.

\bibliography{references}
\end{document}